\documentclass{eect}
\usepackage{amsmath}
  \usepackage{paralist}
  \usepackage{graphics} 
  \usepackage{epsfig} 
\usepackage{graphicx}  \usepackage{epstopdf}
 \usepackage[colorlinks=true]{hyperref}
\hypersetup{urlcolor=blue, citecolor=red}

\def\currentvolume{6}
 \def\currentissue{2}
  \def\currentyear{2017}
   \def\currentmonth{June}
    \def\ppages{X--XX}
     \def\DOI{10.3934/eect.2017009}

\newtheorem{theorem}{Theorem}[section]

\newtheorem{lemma}[theorem]{Lemma}
\newtheorem{proposition}{Proposition}

\theoremstyle{definition}
\newtheorem{definition}[theorem]{Definition}
\newtheorem{remark}{Remark}
\newtheorem*{notation}{Notation}

\title[Logarithmic NLS equation with a $\delta^{\prime}$-interaction] 
      {Stability of  ground states  for logarithmic Schr\"{o}dinger equation with a $\delta^{\prime}$-interaction}

\author[Alex H. Ardila]{}


 \email{alexha@ime.usp.br}

\begin{document}
\maketitle

\centerline{\scshape Alex H. Ardila}
\medskip
{\footnotesize
 \centerline{Department of Mathematics, IME-USP, Cidade Universit\'aria}
  \centerline{CEP 05508-090, S\~ao Paulo, SP, Brazil}
} 

\medskip


\bigskip

 \centerline{(Communicated by Thierry Cazenave)  }

\begin{abstract}
In this paper we study the one-dimensional logarithmic Schr\"{o}din- \break ger equation  perturbed by an attractive $\delta^{\prime}$-interaction
\[
i\partial_{t}u+\partial^{2}_{x}u+ \gamma\delta^{\prime}(x)u+u\, \mbox{Log}\left|u\right|^{2}=0, \quad (x,t)\in\mathbb{R}\times\mathbb{R},
\]
where  $\gamma>0$. We establish the existence and uniqueness of the solutions of the associated Cauchy problem in a suitable functional framework.  In the attractive $\delta^{\prime}$-interaction case, the set of the ground state is completely determined. More precisely: if
$0<\gamma\leq 2$, then there is a single ground state and it is an odd function; if $\gamma>2$, then there exist two non-symmetric ground states. Finally, we show that the ground states are orbitally stable via a variational approach.
\end{abstract}

\section{Introduction}
In this paper we consider the following nonlinear Schr\"{o}dinger equation with a delta prime potential:
\begin{equation}
\label{00NL}
 i\partial_{t}u+\partial^{2}_{x}u+ \gamma\delta^{\prime}(x)u+u\, \mbox{Log}\left|u\right|^{2}=0,
\end{equation}
where  $u=u(x,t)$ is a complex-valued function of $(x,t)\in \mathbb{R}\times\mathbb{R}$. The parameter $\gamma$ is real; when positive, the potential is called attractive, otherwise repulsive. The formal expression $-\partial^{2}_{x}-\gamma\delta^{\prime}(x)$ which appear in \eqref{00NL} admit a precise interpretation as self-adjoint  operator ${H}_{\gamma}$ on $L^{2}(\mathbb{R^{}}^{})$ associated with the quadratic form $\mathfrak{t}_{\gamma}$ (see \cite{ADN1, AKAM}),
\[
\mathfrak{t_{\gamma}}[u,v]=\Re\int_{\mathbb{R}}u^{\prime}(x)\overline{v^{\prime}(x)} dx-\gamma^{-1}\Re \left[\left(u(0+)-u(0-)\right) \overline{\left(v(0+)-v(0-)\right)}\right],
\]
defined on the domain $ \mathrm{dom}(\mathfrak{t_{\gamma}})=H^{1}(\mathbb{R}\setminus\left\{0\right\})$. To be more specific, it is clear that this form is bounded from below and closed on $H^{1}(\mathbb{R}\setminus\left\{0\right\})$.  Then the self-adjoint operator on $L^{2}(\mathbb{R^{}}^{})$ associated with $\mathfrak{t}_{\gamma}$ is  given by  (see \cite{AKAM})
\begin{equation*}
 H_{\gamma}u=-\frac{d^{2}u}{dx^{2}}\quad \text{ for $x\neq0$,}
\end{equation*}
 on the domain
\begin{equation*}
\mathrm{dom}(H_{\gamma})=\left\{u\in  H^{2}(\mathbb{R}\verb'\'  \left\{0 \right\}) :u^{\prime}(0+)=u^{\prime}(0-), u^{}(0+)-u^{}(0-)=-\gamma u^{\prime}(0)\right\}.
\end{equation*}
We note that  $H_{\gamma}$ can also be defined via theory of self-adjoint extensions of symmetric operators (see \cite{APIN}). Moreover, the following spectral properties of  ${H}_{\gamma}$ are known: $\sigma_{\rm ess}({H}_{\gamma})=[0,\infty)$; if $\gamma\leq 0$, then $\sigma_{\rm p}({H}_{\gamma})=\emptyset$; if $\gamma> 0$, then $\sigma_{\rm p}({H}_{\gamma})=\left\{-4/\gamma^{2}\right\}$.

In the absence of the point interaction, this equation is applied in many branches of physics, e.g., quantum optics, nuclear physics, fluid dynamics, geophysics, plasma physics and  Bose-Einstein condensation (see, e.g. \cite{HE123, APLES} and references therein).  The study of stability and instability for one-dimensional singularly perturbed nonlinear Schr\"{o}dinger equations started only a few years ago and is currently  regarded as one of the most interesting research topics in modern nonlinear wave theory (see, e.g. \cite{ADN1, NANAA, ADNP, AESM,AV, AnguloArdila2016, RJJ, FO, KEO, LFF} and references therein). In particular, the nonlinear Schr\"{o}dinger equation with a power nonlinearity $\left|u\right|^{p-1}u$ and a $\delta^{\prime}$-interaction have been studied extensively.  Among such works, let us mention \cite{ADN1, NANAA, ADNP, AESM,AV}.

The nonlinear Schr\"{o}dinger equation \eqref{00NL} is formally associated with the energy functional $E$ define by
\begin{equation*}
E(u)=\frac{1}{2}\mathfrak{t_{\gamma}}[u]-\frac{1}{2}\int_{\mathbb{R}}\left|u\right|^{2}\mbox{Log}\left|u\right|^{2}dx.
\end{equation*}
Unfortunately, due to the singularity of the logarithm at the origin, the functional fails to be finite as well of class $C^{1}$ on $ \mathrm{dom}(\mathfrak{t_{\gamma}})=H^{1}(\mathbb{R}\setminus\left\{0\right\})$. Due  to  this  loss  of  smoothness,  it is convenient  to work in a suitable Banach space endowed with a Luxemburg type norm in order to make functional $E$  well defined and $C^{1}$ smooth.

Indeed, we consider the reflexive Banach space (see Section \ref{S:-1} below)
\begin{equation}\label{ASE}
{W}(\dot{\mathbb{R}})=\bigl\{u\in H^{1}(\mathbb{R}\setminus\left\{0\right\}):\left|u\right|^{2}\mathrm{Log}\left|u\right|^{2}\in L^{1}(\mathbb{R})\bigr\}.
\end{equation}
Then, by Proposition \ref{DFFE} in Section \ref{S:-1} we have that the energy functional $E$ is well-defined and of class $C^1$ on $W(\dot{\mathbb R})$. Moreover, from Lemma \ref{APEX23}, we have that the operator
\begin{equation*}
\begin{cases}
{W}^{}(\dot{\mathbb{R}})\rightarrow {W}^{\prime}(\dot{\mathbb{R}})\\
u\rightarrow {H}_{\gamma}u-u\, \mbox{Log}\left|u\right|^{2}
\end{cases}
\end{equation*}
is continuous and bounded. Here, ${W}^{\prime}(\dot{{\mathbb{R}}})$ is the dual space of ${W}^{}(\dot{{\mathbb{R}}})$. Therefore,  if $u\in C(\mathbb{R}, {W}(\dot{{\mathbb{R}^{}}}))\cap C^{1}(\mathbb{R}, {W}^{\prime}(\dot{{\mathbb{R}}}))$, then equation \eqref{00NL} makes sense in ${W}^{\prime}(\dot{{\mathbb{R}}})$.

The following proposition is concerned with the well-posedness of the Cauchy problem for \eqref{00NL} in the energy space ${W}(\dot{\mathbb{R}})$.
\begin{proposition} \label{PCS}
For any $u_{0}\in {W}(\dot{\mathbb{R}})$, there is a unique maximal solution  $u\in C(\mathbb{R},{W}(\dot{\mathbb{R}^{}}))\cap C^{1}(\mathbb{R}, {W}^{\prime}(\dot{\mathbb{R}^{}}) )$  of \eqref{00NL}  such that $u(0)=u_{0}$ and $\sup_{t\in \mathbb{R}}\left\|u(t)\right\|_{{W}(\dot{\mathbb{R}})}<\infty$. Furthermore, the conservation of energy and charge hold; that is,
\begin{equation*}
E(u(t))=E(u_{0})\quad  and \quad \left\|u(t)\right\|^{2}_{L^{2}}=\left\|u_{0}\right\|^{2}_{L^{2}}\quad  \text{for all $t\in \mathbb{R}$}.
\end{equation*}
\end{proposition}

In this paper, we consider the case $\gamma>0$, and study the variational structure and  orbital stability of the standing wave solution of  \eqref{00NL} of the form  $u(x,t)=e^{i\omega t}\varphi(x)$ where $\omega\in \mathbb{R}$ and $\varphi$ is a real valued function which has to solve the following stationary  problem
\begin{equation}\label{GS}
\begin{cases}
{H}_{\gamma}\varphi+\omega \varphi-\varphi\, \mathrm{Log}\left|\varphi \right|^{2}=0 \quad \mbox{in} \quad {W}^{\prime}(\dot{{\mathbb{R}}}),\\
\varphi\in W(\dot{\mathbb{R}})\setminus\{0\}.
\end{cases}
\end{equation}

More precisely, the main aim of this paper is to analyse the existence and stability of  ground states for one-dimensional logarithmic Schr\"{o}dinger equation \eqref{00NL}. It is important to note that our work was inspired by the recent interesting results of R. Adami and D. Noja \cite{ADNP}. Indeed, the techniques used here are similar in many respects to those used in \cite{ADNP}.

\begin{definition}
For $\gamma>0$ and $\omega\in \mathbb{R}$, we define the following functionals of class $C^{1}$ on $W(\dot{\mathbb{R}})$:
\begin{align*}
 S_{\omega,\gamma}(u)&=\frac{1}{2}\mathfrak{t_{\gamma}}[u]+ \frac{\omega+1}{2}\|u \|^{2}_{L^{2}}-\frac{1}{2}\int_{\mathbb{R}}\left|u\right|^{2}\mathrm{Log}\left|u\right|^{2}dx,\\
 I_{\omega,\gamma}(u)&=\mathfrak{t_{\gamma}}[u]+\omega\,\|u \|^{2}_{L^{2}}-\int_{\mathbb{R}}\left|u\right|^{2}\mathrm{Log}\left|u\right|^{2}dx.
\end{align*}
\end{definition}
Note that \eqref{GS} is equivalent to $S^{\prime}_{\omega,\gamma}(\varphi)=0$, and $I_{\omega,\gamma}(u)=\left\langle S_{\omega,\gamma}^{\prime}(u),u\right\rangle$ is the so-called Nehari functional.

From the physical point viewpoint, an important role is played by the ground state solution of \eqref{GS}. We recall that a solution $\varphi\in{W}(\dot{\mathbb{R}_{}})$ of \eqref{GS} is termed as a ground state if it has some minimal action among all solutions of \eqref{GS}. To be more specific,  we consider the minimization problem
\begin{align}
\begin{split}\label{MPE}
d_{\gamma}(\omega)&={\inf}\left\{S_{\omega,\gamma}(u):\, u\in W(\dot{\mathbb{R}^{}})  \setminus  \left\{0 \right\},  I_{\omega,\gamma}(u)=0\right\} \\
&=\frac{1}{2}\,{\inf}\left\{\left\|u\right\|_{L^{2}}^{2}:u\in  W(\dot{\mathbb{R}^{}}) \setminus \left\{0 \right\},  I_{\omega,\gamma}(u)= 0 \right\},
\end{split}
\end{align}
and define the set of ground states  by
\begin{equation*}
 \mathcal{N}_{\omega,\gamma}=\bigl\{ \varphi\in W(\dot{\mathbb{R}^{}})  \setminus  \left\{0 \right\}: S_{\omega,\gamma}(\varphi)=d_{\gamma}(\omega), \,\, I_{\omega,\gamma}(\varphi)=0\bigl\}.
\end{equation*}
The set $\bigl\{u\in W(\dot{\mathbb{R}^{}}) \setminus  \left\{0 \right\},  I_{\omega,\gamma}(u)=0\bigl\}$ is called the Nehari manifold. Notice that  the above set contains all stationary point of $S_{\omega,\gamma}$.

For $\gamma>0$, the existence of minimizers for \eqref{MPE} is obtained through variational argument. We will show the following theorem in Section \ref{S:2}.
\begin{theorem} \label{ESSW}
Let  $\gamma>0$. There exists a minimizer of $d_{\gamma}(\omega)$ for any $\omega \in \mathbb{R}$; that is, the set of ground states $\mathcal{N}_{\omega,\gamma}$ is not empty.
\end{theorem}
\begin{remark}\label{RM}
Let $u\in \mathcal{N}_{\omega,\gamma}$.  Then, there exists a Lagrange multiplier  $\Lambda\in \mathbb{R}$ such that $S^{\prime}_{\omega,\gamma}( u)=\Lambda I^{\prime}_{\omega,\gamma}(u)$. Thus, we have $\left\langle S^{\prime}_{\omega,\gamma}( u),u\right\rangle=\Lambda\left\langle  I^{\prime}_{\omega,\gamma}(u),u\right\rangle$. The fact that   $\left\langle S^{\prime}_{\omega,\gamma}( u),u\right\rangle =I_{\omega,\gamma}(u)=0$ and $\left\langle  I^{\prime}_{\omega,\gamma}(u),u\right\rangle= -2\left\|u_{}\right\|_{L^{2}}^{2}<0$, implies $\Lambda=0$; that is, $S_{\omega,\gamma}^{\prime}(u)=0$. Therefore,  $u$ is a solution of the stationary problem \eqref{GS}.
\end{remark}

Before proceeding to our main results, we state the following  proposition.
\begin{proposition} \label{11}
Let $\gamma>0$ and  $\omega\in \mathbb{R}$. Any function that belongs to the set of ground states $\mathcal{N}_{\omega,\gamma}$ has the form $e^{i\theta}\phi^{t_{1},t_{2}}_{\omega}$, where $\theta\in\mathbb{R}$,
\begin{equation}\label{Qo}
\phi^{t_{1},t_{2}}_{\omega}(x):=
\begin{cases}
e^{\frac{\omega+1}{2}}e^{-\frac{1}{2}(x+t_{1})^{2}}, &\text{if $x>0$;}\\
-e^{\frac{\omega+1}{2}}e^{-\frac{1}{2}(x-t_{2})^{2}}, &\text{if $ x<0$;}
\end{cases}
\end{equation}
and the couple $(t_{1},t_{2})\in \mathbb{R}^{+}\times\mathbb{R}^{+}$ solves the system
\begin{gather}\label{3S}
\bigg\{
\begin{aligned}
t_{1}e^{-\frac{1}{2}t^{2}_{1}}&=t_{2}e^{-\frac{1}{2}t^{2}_{2}}\\
t^{-1}_{1}+t^{-1}_{2}&=\gamma.
\end{aligned}
\end{gather}
\end{proposition}
Notice that in order to identify  the ground states we must find the solutions of system \eqref{3S}; this solutions will be explicitly calculated in  Section \ref{S:3}. More precisely, for every $0<\gamma\leq 2$ the system \eqref{3S} has exactly one solution  given by $t_{1}={2}{\gamma}^{-1}$, $t_{2}= {2}{\gamma}^{-1}$. At $\gamma=2$ two new solutions arise. Indeed, for $\gamma>2$  the system \eqref{3S} has exactly three solutions and one of them is given by $t_{1}={2}{\gamma}^{-1}$, $t_{2}= {2}{\gamma}^{-1}$. See Proposition \ref{3LL} for more details.

Now we are ready to state our first main result.
\begin{theorem} \label{12}
Let  $\gamma>0$ and $\omega \in \mathbb{R}$. Then the following assertions hold.\\
 (i) If $0<\gamma\leq 2$, then  $\mathcal{N}_{\omega,\gamma}=\bigl\{e^{i\theta}\phi^{t_{\ast},t_{\ast}}_{\omega}: \theta\in\mathbb{R} \bigl\}$, where $t_{\ast}={2}{\gamma}^{-1}$.\\
(ii) If $\gamma>2$, then  $\mathcal{N}_{\omega,\gamma}=\bigl\{e^{i\theta}\phi^{t_{1},t_{2}}_{\omega},\,\,\,e^{i\theta}\phi^{t_{2},t_{1}}_{\omega} :\,\theta\in \mathbb{R} \bigl\}$, where $(t_{1},t_{2})$ is the unique couple, with $t_{1}<{2}{\gamma}^{-1}<t_{2}$, that solves the system \eqref{3S}.
\end{theorem}
A careful consideration of this theorem reveals the presence of a branch of ground states, that, at the critical value $\gamma=2$, bifurcates
in two branches; correspondingly, parity symmetry is broken. The occurrence of bifurcation and spontaneous symmetry breaking phenomenon in the ground state has been investigated in \cite{JACWE04} and more recently in \cite{KKDP, AS009, FUKSACC}.

The next step in the study of ground states to \eqref{GS} is to understand their stability. The basic symmetry associated to equation \eqref{00NL} is the phase-invariance (while the translation invariance  does not hold due to the defect). Thus,  the definition  of stability  takes into account only  this  type of symmetry and is formulated as follows.

\begin{definition}\label{2D111}
We say that  a standing wave solution $u(x,t)=e^{i\omega t}\phi(x)$ of \eqref{00NL} is orbitally stable in $W(\dot{\mathbb{R}})$ if for any  $\epsilon>0$ there exist $\eta>0$  such that if $u_{0}\in W(\dot{\mathbb{R}})$ and $\left\|u_{0}-\varphi \right\|_{W(\dot{\mathbb{R}})}<\eta$, then the solution $u(t)$ of  \eqref{00NL}  with $u(0)=u_{0}$ exist for all $t\in \mathbb R$ and satisfies
\begin{equation*}
{\rm\sup\limits_{t\in \mathbb R}} {\rm\inf\limits_{\theta\in \mathbb{R}}} \|u(t)-e^{i\theta}\phi \|_{W(\dot{\mathbb{R}})}<\epsilon.
\end{equation*}
Otherwise, the standing wave $e^{i\omega t}\phi(x)$ is said to be  unstable in $W(\dot{\mathbb{R}})$.
\end{definition}

Our second  main result shows that the ground states are orbitally stable for every $\omega\in \mathbb{R}$.

\begin{theorem} \label{EST}
Let  $\gamma>0$ and $\omega \in \mathbb{R}$. Then the following assertions hold.\\
 (i) Let $0<\gamma\leq 2$, then the standing wave $e^{i\omega t}\phi^{t_{\ast},t_{\ast}}_{\omega}(x)$ is orbitally stable in  $W(\dot{\mathbb{R}})$.\\
 (ii) Let $\gamma>2$, then the standing waves $e^{i\omega t}\phi^{t_{2},t_{1}}_{\omega}(x)$ and $e^{i\omega t}\phi^{t_{1},t_{2}}_{\omega}(x)$ are orbitally stable in $W({\dot{\mathbb{R}}})$.
\end{theorem}
The proof of Theorem \ref{EST} is based on the variational characterization of the stationary solutions $\varphi$ for \eqref{GS} as minimizers  of the action $S_{\omega, \gamma}$ on the Nehari manifold (see Theorem \ref{12}) and from the compactness of the minimizing
sequences (see Lemma \ref{CSM} below) for $d_{\gamma}(\omega)$.
We remark that nothing is known about orbital stability of the first excited state arising from the ground state from the bifurcation point, which exist for every $\gamma>2$. It is a conjecture  that excited states are unstable, but we not have a proof of this fact.

\begin{remark}
A similar analysis is carried out in \cite{AnguloArdila2016}, in the case of  logarithmic Schr\"{o}dinger equation with  attractive delta potential. Indeed, it was shown in \cite{AnguloArdila2016} that there exists a unique positive (up to a phase) ground state and it is orbitally stable.
\end{remark}

The rest of the paper is organized as follows. In Section \ref{S:-1}, we analyse the structure of the energy space $\dot{W}(\mathbb{R}^{})$.
In Section \ref{S:0}, we give an idea of the proof of  Proposition \ref{PCS}. In Section \ref{S:2} we prove, by variational techniques, the existence of a minimizer for $d_{\gamma}(\omega)$. In Section \ref{S:3}, we explicitly compute the ground states (Theorem \ref{12}). The Section \ref{S:4} is devoted to the proof of Theorem \ref{EST}. In the Appendix we list some properties of the Orlicz space $L^{A}(\mathbb{R})$ defined in Section \ref{S:-1}.

\begin{notation}
The space $L^{2}(\mathbb{R},\mathbb{C})$  will be denoted  by $L^{2}(\mathbb{R})$ and its norm by $\|\cdot\|_{L^{2}}$.  This space will be endowed  with the real scalar product
\begin{equation*}
\left(u,v\right)=\Re\int_{\mathbb{R}}u\overline{v}\, dx \,\,\,\,\,\,\,\,\,\,\,\,\, \mathrm{for }\,\,\,\,\,\,\,\ u,v\in L^{2}\left(\mathbb{R}\right).
\end{equation*}
The space $H^{1}(\mathbb{R},\mathbb{C})$ will be denoted by $H^{1}(\mathbb{R})$, its norm by $\|\cdot\|_{H^{1}(\mathbb{R})}$. We write $H_{\rm rad}^{1}(\mathbb{R})$ for the space of radial (even) function on $H^{1}(\mathbb{R})$. The space $H^{1}(\mathbb{R}\setminus\left\{0\right\},\mathbb{C})$ is equipped with their usual real inner product, it will be denoted  by $\Sigma$ and its norm by $\|\cdot\|_{\Sigma}$. We denote by $C_{0}^{\infty}\left(\mathbb{R}\setminus\left\{0\right\} \right)$ the set of $C^{\infty}$ functions from $\mathbb{R}\setminus\left\{0\right\}$ to $\mathbb{C}$ with compact support. $\left\langle \cdot , \cdot \right\rangle$ is the duality pairing between $E^{\prime}$ and $E$, where $E$ is a Hilbert (more generally, Banach space) and $E^{\prime}$ is its dual. Characteristic function  on $\mathbb{R}^{+}=(0,+\infty)$ (resp. $\mathbb{R}^{-}=(-\infty,0)$) will be denoted by $\chi_{+}$ (resp. $\chi_{-}$). Throughout this paper, the letter $C$ will denote positive constants.
\end{notation}

\section{Preliminaries}\label{S:-1}

The purpose of this section is to describe  the structure of space ${W}(\dot{\mathbb{R}})$.  Also, we will show that the energy functional $E$ is of class $C^{1}$ on ${W}(\dot{\mathbb{R}})$.

We need to introduce some notation. Define
\begin{equation*}
F(z)=\left|z\right|^{2}\mbox{Log}\left|z\right|^{2}\,\,\,\,\,\,\,\,  \text{for every  $z\in\mathbb{C}$},
\end{equation*}
and as in \cite{CL},  we define the functions  $A$, $B$ on $\left[0, \infty\right)$  by
\begin{equation}\label{IFD}
A(s)=
\begin{cases}
-s^{2}\,\mbox{Log}(s^{2}), &\text{if $0\leq s\leq e^{-3}$;}\\
3s^{2}+4e^{-3}s^{}-e^{-6}, &\text{if $ s\geq e^{-3}$;}
\end{cases}
\,\,\,\,\,\,\,\,\,  B(s)=F(s)+A(s).
\end{equation}
Furthermore, let be functions $a$, $b$, defined by
\begin{equation}\label{abapex}
a(z)=\frac{z}{|z|^{2}}\,A(\left|z\right|)\,\, \text{ and  }\,\, b(z)=\frac{z}{|z|^{2}}\,B(\left|z\right|) \text{  for $z\in \mathbb{C}$, $z\neq 0$}.
\end{equation}
Notice that we have $b(z)-a(z)=z\,\mathrm{Log}\left|z\right|^{2}$. It follows that $A$ is a nonnegative  convex and increasing function, and $A\in C^{1}\left([0,+\infty)\right)\cap C^{2}\left((0,+\infty)\right)$. The Orlicz space $L^{A}(\mathbb{R})$ corresponding to $A$ is defined by
\begin{equation*}
L^{A}(\mathbb{R})=\left\{u\in L^{1}_{loc}(\mathbb{R}) : A(\left|u\right|)\in L^{1}_{}(\mathbb{R})\right\},
\end{equation*}
equipped with the Luxemburg norm
\begin{equation*}
\left\|u\right\|_{L^{A}}={\inf}\left\{k>0: \int_{\mathbb{R}^{}}A\left(k^{-1}{\left|u(x)\right|}\right)dx\leq 1 \right\}.
\end{equation*}
Here as usual $L^{1}_{loc}(\mathbb{R})$ is the space of all locally Lebesgue integrable functions. It is proved in \cite[Lemma 2.1]{CL} that $A$ is a Young-function which is $\Delta_{2}$-regular and $\left(L^{A}(\mathbb{R}),\|\cdot\|_{L^{A}} \right)$ is a separable reflexive  Banach space.

Next, we consider the reflexive Banach space $W(\dot{\mathbb{R}})=\Sigma\cap L^{A}(\mathbb{R})$ equipped with the usual norm $ \left\|u\right\|_{{W}(\dot{\mathbb{R}})}=\left\|u\right\|_{\Sigma}+\left\|u\right\|_{L^{A}}$ (see \eqref{ASE}). It is easy to see that ${W}(\dot{\mathbb{R}})=\bigl\{u\in \Sigma:\left|u\right|^{2}\mathrm{Log}\left|u\right|^{2}\in L^{1}(\mathbb{R})\bigr\}$  (see \cite[Proposition 2.2]{CL} for more details). Furthermore, it is clear that the dual space (see \cite[Proposition 1.1.3]{CB})
\begin{equation*}
{W}^{\prime}(\dot{\mathbb{R}})=\Sigma^{\prime}+L^{A^{\prime}}(\mathbb{R}^{}),
\end{equation*}
where the Banach space ${W}^{\prime}(\dot{\mathbb{R}})$ is equipped with its usual norm. Here, $L^{A^{\prime}}(\mathbb{R}^{})$ is the dual space of $L^{A}(\mathbb{R})$ (see \cite{CL}).

\begin{lemma} \label{APEX23}
The operator $L: u\rightarrow H_{\gamma}u-u\,  \mathrm{Log}\left|u\right|^{2}$ is continuous from  $W(\dot{\mathbb{R}})$  to $W^{\prime}(\dot{\mathbb{R}})$. The image under $L$ of a bounded subset of $W(\dot{\mathbb{R}})$ is a bounded subset of $W^{\prime}(\dot{\mathbb{R}})$.
\end{lemma}
\begin{proof}
Notice that, as usual, the operator $H_{\gamma}$ is naturally extended to $H_{\gamma}: \Sigma\rightarrow\Sigma^{\ast}$ defined by
\begin{equation*}
\left\langle {H}_{\gamma}u,v\right\rangle=\mathfrak{t_{\gamma}}[u_{},v],  \quad  \textrm{for} \quad u,v\in \Sigma.
\end{equation*}
Now, using ${W}({\dot{\mathbb{R}}})\hookrightarrow \Sigma$, we obtain that $u\rightarrow H_{\gamma}u$ is continuous from ${{W}}(\dot{{\mathbb{R}}})$ to ${W}^{\prime}(\dot{{\mathbb{R}}})$. On the other hand, one easily verifies that for $\epsilon>0$ there exist $C_{\epsilon}$ such that
\begin{equation*}
\left|b(z)-b(z_{1})\right|\leq C_{\epsilon}\left|z-z_{1}\right|(\left|z\right|+\left|z_{1}\right|)^{\epsilon}\quad \text{  for $z$, $z_{1}\in \mathbb{C}$},
\end{equation*}
which combined with H{\"o}lder inequality and Sobolev embedding gives
\begin{equation*}
\left\|b(u)-b(v)\right\|^{}_{L^{2}}\leq C\left\|u-v\right\|_{\Sigma}\left(\left\|u\right\|_{\Sigma}+\left\|u\right\|_{\Sigma}\right)^{1/2}   \quad      \text{  for $u$, $v\in \Sigma$}.
\end{equation*}
Thus, we obtain that $u\rightarrow b(u)$ is continuous and bounded from $\Sigma$ to $L^{2}(\mathbb{R})$, then from ${{W}}(\dot{{\mathbb{R}}})$ to ${W}^{\prime}(\dot{{\mathbb{R}}})$. Finally, by \cite[Lemma 2.3]{CL}, $u\rightarrow a(u)$ is continuous and bounded from $L^{A}(\mathbb{R})$ to $L^{A^{\prime}}(\mathbb{R}^{})$, then from ${{W}}(\dot{{\mathbb{R}}})$ to ${W}^{\prime}(\dot{{\mathbb{R}}})$ and since $b(z)-a(z)=z\mathrm{Log}\left|z\right|^{2}$, Lemma \ref{APEX23} is proved.
\end{proof}

The following proposition shows that $E\in C^{1}(W(\dot{\mathbb{R}}), \mathbb{R})$. More exactly, we obtain the following result.

\begin{proposition} \label{DFFE}
The operator $E: W(\dot{\mathbb{R}})\rightarrow \mathbb R$  is of class $C^{1}$ and  for $u\in W(\dot{\mathbb{R}})$ the  Fr\'echet derivative of $E$ in $u$ exists and it is given by
\begin{equation*}
E^{\prime}(u)={H}_{\gamma}u-u\, \mathrm{Log}\left|u\right|^{2}-u.
\end{equation*}
\end{proposition}
\begin{proof}
We first show that $E$  is continuous. Notice that
\begin{equation}\label{CCC}
E(u)=\frac{1}{2}\mathfrak{t}_{\gamma}[u]+\frac{1}{2}\int_{\mathbb{R}}A(\left|u\right|)dx-\frac{1}{2}\int_{\mathbb{R}}B(\left|u_{}\right|)dx.
\end{equation}
The first term in the right-hand side of \eqref{CCC} is continuous $\Sigma\rightarrow \mathbb R$, and it follows from
Proposition \ref{orlicz}(i) in Appendix that the second term is continuous $L^{A}(\mathbb{R})\rightarrow \mathbb{R}$. Moreover, by \eqref{DB} below, we get that the third term  right-hand side of \eqref{CCC} is continuous $\Sigma\rightarrow \mathbb R$. Therefore, $E\in C(W(\dot{\mathbb{R}}),\mathbb{R})$. Now, direct calculations show that, for $u$, $v\in W(\dot{\mathbb{R}})$, $t\in (-1,1)$ (see \cite[Proposition 2.7]{CL}),
\begin{equation*}
\lim_{t\rightarrow 0} \frac{E(u+tv)-E(u)}{t}=\bigl\langle {H}_{\gamma}u-u\, \mbox{Log}\left|u\right|^{2}-u,v\bigl\rangle
\end{equation*}
Thus, $E$ is G\^ateaux differentiable. Then, by Lemma \ref{APEX23} we see that $E$ is  Fr\'echet differentiable  and $E^{\prime}(u)={H}_{\gamma}u-u\,\mbox{Log}\left|u\right|^{2}-u$.
\end{proof}

\section{The Cauchy problem}
\label{S:0}
In this section we sketch the proof of the global well-posedness of the Cauchy Problem  for \eqref{00NL} in the energy space  $\dot{W}(\mathbb{R}^{})$. The proof of Proposition \ref{PCS} is an adaptation of the proof of \cite[Theorem 9.3.4]{CB} (see also \cite{AnguloArdila2016}).  So, we will approximate the logarithmic nonlinearity by a smooth nonlinearity, and as a consequence we construct a sequence of global solutions of the regularized Cauchy problem in $C(\mathbb{R},\Sigma)\cap C^{1}(\mathbb{R}, \Sigma^{\prime})$,  then we pass to the limit using standard compactness results, extract a subsequence which converges to the solution of the limiting equation \eqref{00NL}.

Before proceeding to the proof of Proposition \ref{PCS}, we first need some preliminary remarks. Let us recall that $\Sigma=H^{1}(\mathbb{R})\oplus\mbox{span}\left\{\zeta\right\}$, where
\begin{equation}\label{Z}
\zeta(x)=\frac{1}{2}\left(\frac{x}{\left|x\right|}\right)e^{-\left|x\right|},\quad  \textrm{for} \quad  x\in\mathbb{R}\verb'\'  \left\{0 \right\}.
\end{equation}
Moreover,  it is known that for any function $u\in\Sigma$ there exists a unique couple of functions $u_{-}$, $u_{+}\in H_{\rm rad}^{1}(\mathbb{R})$  such that $u(x)=\chi_{+}(x)u_{+}(x)+\chi_{-}(x)u_{-}(x)$ for all $x\in \mathbb{R}\setminus\left\{0\right\}$. As a consequence (see \cite{ADN1, ADNP}),
\begin{equation}\label{ADE}
\left\|u\right\|^{2}_{\Sigma}=\frac{1}{2}\left\|u_{+}\right\|^{2}_{ H^{1}_{}(\mathbb{R})}+\frac{1}{2}\left\|u_{-}\right\|^{2}_{ H^{1}_{}(\mathbb{R})}.
\end{equation}

Next, we regularize the logarithmic nonlinearity near the origin.  For $z\in \mathbb{C}$ and $m\in \mathbb{N}$,  we define the functions $a_{m}$ and $b_{m}$ by
\begin{equation*}
a_{m}(z)=
\begin{cases}
 a_{}(z), &\text{if $\left|z\right|\geq \frac{1}{m}$;}\\
m\,z\,a_{}(\frac{1}{m}) , &\text{if $\left|z\right|\leq \frac{1}{m}$;}
\end{cases}
\quad \text{and} \quad
b_{m}(z)=
\begin{cases}
 b_{}(z) , &\text{if $\left|z\right|\leq {m}$;}\\
\frac{z}{m}\,b({m}) , &\text{if $\left|z\right|\geq {m}$,}
\end{cases}
\end{equation*}
where   $a$ and $b$ were defined in \eqref{abapex}. Moreover,  we set $g_{m}(z)=b_{m}(z)-a_{m}(z)$ and
\begin{equation*}
G_{m}(z)=\int^{\left|z\right|}_{0}g_{m}(s)ds
\end{equation*}
for  $z\in \mathbb{C}$. Notice that  the function $g_{m}: L^{2}(\mathbb{R})\rightarrow L^{2}(\mathbb{R})$ is globally Lipschitz continuous  on $ L^{2}(\mathbb{R})$ and $(g_{m}(u),iu)=0$ for every $u\in \Sigma$.

For the proof of Proposition \ref{PCS}, we will use the following two results.
\begin{proposition} \label{APCS}
For any $u_{0}\in \Sigma$ and $m\in \mathbb{N}$, there is a unique maximal solution  $u^{m}\in C(\mathbb{R},\Sigma)\cap C^{1}(\mathbb{R}, \Sigma^{\prime})$  of
\begin{equation}\label{AHAX}
\begin{cases}
i\partial_{t}u^{m}-H_{\gamma}u^{m}+g_{m}(u^{m})=0,\\
u^{m}(0)=u_{0}.
\end{cases}
\end{equation}
Furthermore, the conservation of charge and energy hold; that is, for all $t\in \mathbb{R}$, $\left\|u^{m}(t)\right\|^{2}_{L^{2}}=\left\|u_{0}\right\|^{2}_{L^{2}}$ and
\begin{equation*}
\mathcal{E}_{m}(u^{m}(t))=\mathcal{E}(u_{0}), \quad \text{where} \quad \mathcal{E}_{m}(u)=\frac{1}{2}\mathfrak{t_{\gamma}}[u]-\int_{\mathbb{R}}G_{m}(u)\,dx.
\end{equation*}
\end{proposition}
\begin{proof}
We use the argument in \cite[Proposition 3]{FO} and we apply Theorem 3.7.1 in  \cite{CB}.
First, we note that $H_\gamma$ satisfies $H_\gamma\geq -c$, where $c = 4/\gamma^2$ if $\gamma > 0$, and $c = 0$ if $\gamma < 0$. Thus, $A =- H_\gamma-c$ is a self-adjoint negative operator on $X = L^2(\mathbb{R})$  with domain  $\mathrm{dom}(A) = \mathrm{dom}(H_\gamma)$. In addition,  it is not difficult to show that the norm
\begin{equation*}
\left\|u\right\|^2_{X_A}=\mathfrak{t_{\gamma}}[u]+(c+1)\left\|u\right\|_{L^2}^2,
\end{equation*}
is  equivalent to the usual $\Sigma$-norm. Moreover, it is easy to see that the conditions (3.7.1), (3.7.3)-(3.7.6) in \cite[Section 3.7]{CB} hold choosing $r=\rho=2$, since we are in one dimensional case. Also, the condition (3.7.2) with $p=2$ follows easily from the self-adjointness of $A$. We remark that only the case $p=2$ in (3.7.2) is needed for our case since we can take $r=\rho=2$. Notice that the uniqueness of  solutions follows from Gronwall's lemma (see \cite[Corollary 3.3.11]{CB}). Finally, from \cite[Corollary 3.5.2]{CB} we see that the solution is global and uniformly bounded in $\Sigma$.
\end{proof}

For $k>0$,  we introduce the Hilbert space $\Sigma_{k}=H_{0}^{1}(B_{k})\oplus\mbox{span}\left\{\zeta\right\}$, where $B_{k}=\left\{ x\in\mathbb{R}: \left|x\right|<k \right\}$ and the function $\zeta$  is defined in \eqref{Z}. It follows in particular that the inclusion map $\Sigma_{k}\hookrightarrow \Sigma$ is continuous.


\begin{lemma}\label{3ACS}
Let $\left\{u^{{m}}\right\}_{m\in\mathbb{N}}$ be a bounded sequence in $L^{\infty}(\mathbb{R},\Sigma)$ and in $W^{1,\infty}(\mathbb{R},\Sigma_{k}^{\prime})$ for $k\in\mathbb{N}$. Then there exists a subsequence, which we still denote by $\left\{u^{{m}}\right\}_{m\in\mathbb{N}}$, and there exist  $u\in L^{\infty}(\mathbb{R},\Sigma)\cap W^{1,\infty}(\mathbb{R},\Sigma_{k}^{\prime})$ for every $k\in\mathbb{N}$, such that the following properties hold:\\
 (i)  $u^{m}(t)\rightharpoonup u^{}(t)$  in  $\Sigma$  as $m\rightarrow \infty$ for every $t\in \mathbb{R}$.\\
 (ii)	For every $t\in\mathbb{R}$ there exists a subsequence  $m_{j}$ such that $u^{m_{j}}(x,t)\rightarrow u^{}(x,t)$   as $j\rightarrow \infty$, for a.a. $x\in \mathbb{R}$.\\
 (iii)	  $u^{m_{}}(x,t)\rightarrow u^{}(x,t)$ as $m\rightarrow \infty$,  for a.a.  $(x,t)\in \mathbb{R}\times\mathbb{R}$.
\end{lemma}
\begin{proof}
We just sketch the proof since it follows the same ideas as the proof  of Lemma 9.3.6 in \cite{CB}. In fact, fix $k\in\mathbb{N}$. Note that
$\left\{u^{m}\right\}$ is a bounded sequence of $ L^{\infty}(B_{k},\Sigma)\cap W^{1,\infty}(B_{k},\Sigma_{k}^{\prime})$. Therefore, by \cite[Proposition 1.1.2]{CB} there exists a subsequence, which we still denote by $\left\{u^{{m}}\right\}_{m\in\mathbb{N}}$, and there exist  $u\in L^{\infty}(\overline{B}_{k},\Sigma)$ such that $u^{m}(t)\rightharpoonup u^{}(t)$  in  $\Sigma$  as $m\rightarrow \infty$ for every $t\in \overline{B}_{k}$. Thus, considering a diagonal sequence, we see that $u^{m}(t)\rightharpoonup u^{}(t)$  in  $\Sigma$  as $m\rightarrow \infty$ for every $t\in \mathbb{R}$ and $u\in L^{\infty}(\mathbb{R},\Sigma)$, and (i) follows.  In addition, by \cite[ Remark 1.3.13(ii)]{CB} and (i), we have that
$u\in W^{1,\infty}(\mathbb{R},\Sigma_{k}^{\prime})$ for every $k\in\mathbb{N}$. The remainder of the proof follows similarly to the remainder of the proof  of \cite[Lemma 9.3.6]{CB}.
\end{proof}

\noindent\begin{proof}[ \it {Proof of Proposition \ref{PCS}}]
 We only  discuss the  modifications that are not sufficiently clear. Applying Proposition \ref{APCS}, we see that  for every $m\in \mathbb{N}$ there exists a unique global solution $u^{m}\in C(\mathbb{R}, \Sigma)\cap C^{1}(\mathbb{R},\Sigma^{\prime})$ of \eqref{AHAX}, which satisfies
\begin{equation}\label{JKL}
\mathcal{E}_{m}(u^{m}(t))=\mathcal{E}_{m}(u_{0})\quad \mbox{and}\quad\left\|u^{m}(t)\right\|^{2}_{L^{2}}=\left\|u_{0}\right\|^{2}_{L^{2}}  \,\, \text{ for all $t\in \mathbb{R}$},
\end{equation}
where \begin{equation*}
\mathcal{E}_{m}(u)=\frac{1}{2}\mathfrak{t_{\gamma}}[u] +\frac{1}{2}\int_{\mathbb{R}}\Phi_{m}(\left|u_{}\right|)dx-\frac{1}{2}\int_{\mathbb{R}}\Psi_{m}(\left|u_{}\right|)dx,
\end{equation*}
and the functions $\Phi_{m}$ and $\Psi_{m}$ are defined by
\begin{equation*}
\Phi_{m}(z)=\frac{1}{2}\int^{\left|z\right|}_{0}a_{m}(s)ds\,\,\,\ \mbox{and} \,\,\,\ \Psi_{m}(z)=\frac{1}{2}\int^{\left|z\right|}_{0}b_{m}(s)ds.
\end{equation*}
It follows from \eqref{JKL} that $u^{m}$ is bounded in $L^{\infty}(\mathbb{R},L^{2}(\mathbb{R}))$. Moreover, we see that  $u^{m}$ is bounded in $L^{\infty}(\mathbb{R},\Sigma)$  (see proof of Step 2 of \cite[Theorem 9.3.4]{CB}). Now, by elementary computations we can check that
 for every  $\epsilon >0$ there exists $C_{\epsilon}$ such that
\begin{equation*}
\left|g_{m}(u)\right|\leq C_{\epsilon}\left(\left|u\right|^{1+\epsilon}+\left|u\right|^{1-\epsilon}\right),
\end{equation*}
which combined with H{\"o}lder and Sobolev embedding gives that $\left\{g_{m}(u^{m})\right\}$ is bounded $L^{\infty}(\mathbb{R},{L^{2}(B_{k}))}\cap L^{\infty}(\mathbb{R},L^{\infty}(\mathbb{R}))$  for $k\in\mathbb{N}$. In particular, $\left\{g_{m}(u^{m})\right\}_{m\in\mathbb{N}}$ is bounded in ${L^{\infty}(\mathbb{R},\Sigma^{\prime}_{k})}$ and thus from \eqref{AHAX} we see that $\left\{u^{{m}}\right\}_{m\in\mathbb{N}}$ is bounded in $W^{1,\infty}(\mathbb{R},\Sigma_{k}^{\prime})$ for every $k\in\mathbb{N}$.  Therefore, we have that  $\left\{u^{{m}}\right\}_{m\in\mathbb{N}}$ satisfies the assumptions of  Lemma \ref{3ACS}. Let $u$ be its limit. It follows from \eqref{AHAX}, that $u^{m}$ satisfies
\begin{equation*}
\int_{\mathbb{R}}\left\langle i\, u^{m}_{t}-{H}_{\gamma}u^{m}+g_{m}(u^{m}),\psi\right\rangle_{} \phi(t)\,dt=0,
\end{equation*}
for every $\psi\in \Sigma_{k}$  and  every $\phi\in C^{\infty}_{c}(\mathbb{R})$.  This means that
\begin{equation}\label{3DPL}
-\int_{\mathbb{R}}\left\langle i\, u^{m}_{}, \psi\right\rangle \phi^{\prime}(t)\,dt+\int_{\mathbb{R}}\mathfrak{t_{\gamma}}[u^{m},\psi] \phi^{}(t)\,dt+\int_{\mathbb{R}}\int_{\mathbb{R}}g_{m}(u^{m})\psi\phi\,dx\,dt=0.
\end{equation}
It follows from property (i) of Lemma \ref{3ACS} that
\begin{equation*}
\lim_{n\rightarrow \infty}\int_{\mathbb{R}}\left[-\left\langle i\, u^{m}_{}, \psi\right\rangle \phi^{\prime}(t)+\mathfrak{t_{\gamma}}[u^{m},\psi] \phi^{}(t)\right]\,dt=\int_{\mathbb{R}}\left[\left\langle - i\, u^{}_{}, \psi\right\rangle \phi^{\prime}(t)+\mathfrak{t_{\gamma}}[u^{},\psi] \phi^{}(t)\right]dt.
\end{equation*}
Next, let $h_{m}(x,t)=g_{m}(u^{m})\psi(x)\phi(t)$.  One can easily see that, by the dominated convergence theorem,  $h_{m}\rightarrow (u\, \mbox{Log}\left|u\right|^{2})\psi\phi$ in $L^{1}(\mathbb{R}\times \mathbb{R})$. Moreover, using  \eqref{3DPL}  we  obtain
\begin{equation}\label{ert}
-\int_{\mathbb{R}}\left\langle i\, u^{}_{}, \psi\right\rangle \phi^{\prime}(t)\,dt+\int_{\mathbb{R}}\mathfrak{t_{\gamma}}[u^{},\psi] \phi^{}(t)\,dt+\int_{\mathbb{R}}\int_{\mathbb{R}} u\, \mbox{Log}\left|u\right|^{2}\psi\phi\,dx\,dt=0.
\end{equation}
Since  $u\in{L^{\infty}(\mathbb{R},L^{A}(\mathbb{R}))}$ (see proof of Step 3 of \cite[Theorem 9.3.4]{CB}) we see that $u\in{L^{\infty}(\mathbb{R},W(\dot{\mathbb{R}}))}$, so that $u\in{W^{1,\infty}(\mathbb{R},W^{\prime}(\dot{\mathbb{R}}))}$ by \eqref{ert} and Lemma \ref{APEX23}. In particular, it follows from \eqref{ert},   that for all $t\in \mathbb{R}$,
\begin{equation*}
i\partial_{t}u-{H}_{\gamma}u+u\, \mbox{Log}\left|u\right|^{2}=0 \quad \mbox{in $W^{\prime}(\dot{\mathbb{R}})$}.
\end{equation*}
In addition, $u(0)=u_{0}$ by property (i) of Lemma \ref{3ACS}. Thus, we obtain that there is a solution $u\in{L^{\infty}(\mathbb{R},{W(\dot{\mathbb{R}})})}\cap W^{1,\infty}(\mathbb{R},{W^{\prime}(\dot{\mathbb{R}})})$ of \eqref{00NL} with $u(0)=u_{0}$. Moreover, arguing in the same way as in the proof of the Step 3 of \cite[Theorem 9.3.4]{CB} we deduce that
\begin{equation}
E(u(t))\leq E(u_{0}) \quad \mbox{and} \quad \left\|u^{}(t)\right\|^{2}_{L^{2}}=\left\|u_{0}\right\|^{2}_{L^{2}} \quad \text{for all} \,\, t\in \mathbb{R}.
\end{equation}
On the other hand, let $u$ and $v$ be two solutions   of \eqref{00NL} in that class. On taking the difference of the two equations and taking the $W(\dot{\mathbb{R}})-W^{\prime}(\dot{\mathbb{R}})$ duality product with $i(u-u)$, we see that
\begin{equation*}
\left\langle u_{t}-v_{t}, u-v\right\rangle=-\Im \int_{\mathbb{R}}\left(u\mathrm{Log}\left|u \right|^{2}-v\mathrm{Log}\left|v \right|^{2} \right)(\overline{u}-\overline{v})dx.
\end{equation*}
Thus, from \cite[Lemma 9.3.5]{CB} we obtain
\begin{equation*}
\left\|u(t)-v(t)\right\|^{2}_{L^{2}}\leq 8\int^{t}_{0}\left\|u(s)-v(s)\right\|^{2}_{L^{2}}ds.
\end{equation*}
Therefore, the uniqueness of the solution follows by Gronwall's Lemma. In particular, by uniqueness of solution, we deduce the conservation of energy. Later statements can be proved along the same lines as in the proof of Step 4 of \cite[Theorem 9.3.4]{CB}. Finally, the inclusion $u\in C(\mathbb{R}, W(\dot{\mathbb{R}}))\cap C^{1}(\mathbb{R}, W^{\prime}(\dot{\mathbb{R}}))$ follows from conservation laws.
\end{proof}

\section{Existence of a ground state}
\label{S:2}
This section is devoted to the proof of Theorem \ref{ESSW}.

We have divided the proof into a sequence of lemmas. Firstly we give a lemma that extends the  one-dimensional logarithmic Sobolev inequality to the space $\Sigma$.
\begin{lemma} \label{L1}
Let $u$ be any function in $\Sigma$ and $\alpha$ be any positive number. Then
\begin{equation}\label{LSA}
\int_{\mathbb{R}}\left|u(x)\right|^{2}\mathrm{Log}\left|u(x)\right|^{2}dx\leq \frac{\alpha^{2}}{\pi} \|u^{\prime} \|^{2}_{L^{2}}+\left(\mathrm{Log} (2\|u \|^{2}_{L^{2}})-\left(1+\mathrm{Log}\,\alpha\right)\right)\|u \|^{2}_{L^{2}}.
\end{equation}
\end{lemma}
\begin{proof}
The lemma  follows immediately from the standard  logarithmic Sobolev inequality on $H^{1}(\mathbb{R})$ (see \cite[Theorem 8.14]{ELL}) and the decomposition in \eqref{ADE}.
\end{proof}
\begin{lemma}\label{L2}
Let $\gamma>0$ and $\omega\in \mathbb{R}$. Then, the quantity $d_{\gamma}(\omega)$ is positive and satisfies
\begin{equation}\label{EA}
d_{\gamma}(\omega)\geq  \frac{1}{4}\sqrt{\frac{\pi}{2}}e^{\omega+1}e^{-\frac{8}{\gamma ^{2}}}.
\end{equation}
\end{lemma}
\begin{proof}
Let $u\in W(\dot{\mathbb{R}^{}}) \setminus  \left\{0 \right\}$ be such that  $I_{\omega,\gamma}(u)=0$. By H\"{o}lder and Sobolev inequalities  and the decomposition in \eqref{ADE}, we have
\begin{align}
\frac{1}{{\gamma}}\left|u(0+)-u(0-)\right|^{2}&\leq \frac{2}{{\gamma}}\left\{\left|u(0+)\right|^{2}+\left|u(0-)\right|^{2} \right\}\nonumber\\
&\leq \frac{1}{{\gamma}}\left\{\frac{4}{\gamma}\left[\|u_{+}\|^{2}_{L^{2}}+\|u_{-}\|^{2}_{L^{2}}\right] +\frac{\gamma}{4}\left[\|u^{\prime}_{+}\|^{2}_{L^{2}}+\|u^{\prime}_{-}\|^{2}_{L^{2}}\right]\right\}\nonumber\\
&=\frac{8}{\gamma^{2}}\left\|u\right\|^{2}_{L^{2}}+\frac{1}{2}\left\|u^{\prime}\right\|^{2}_{L^{2}}.\label{EA2}
\end{align}
Now, from  $I_{\omega,\gamma}(u)=0$, the logarithmic Sobolev inequality \eqref{LSA} with $\alpha=\sqrt{\frac{\pi}{2}}$ and \eqref{EA2}, it follows
\begin{equation*}
 \left(\omega+1+\mbox{Log}\left(\sqrt{\frac{\pi}{2}}\right)-\frac{8}{\gamma ^{2}}\right)\left\|u\right\|^{2}_{L^{2}}\leq \left(\mbox{Log}\left(2\left\|u\right\|^{2}_{L^{2}}\right) \right)\left\|u\right\|^{2}_{L^{2}}.
\end{equation*}
Thus, $\left\|u\right\|^{2}_{L^{2}}\geq \frac{1}{2}\sqrt{\frac{\pi}{2}} e^{\omega+1}e^{-\frac{8}{\gamma ^{2}}}$. Finally, by the definition of $d_{\gamma}(\omega)$ given in \eqref{MPE}, we get \eqref{EA}.
\end{proof}

Before stating our next lemma we recall a well-known result on the logarithmic Schr\"{o}dinger equation in the absence of the delta prime potential: namely, the  set of solutions of stationary problem
\begin{equation}
 -\partial^{2}_{x} \varphi+\omega \varphi-\varphi\,\mathrm{Log}\left|\varphi \right|^{2}=0, \quad \text{ $x\in\mathbb{R}$,\, $\omega\in\mathbb{R}$,\, $\varphi\in H^{1}(\mathbb{R})\cap L^{A}(\mathbb{R})$},
\end{equation}
is given by $\bigl\{e^{i\theta}\phi_{\omega}(\cdot-y); \theta\in\mathbb{R},  y\in\mathbb{R} \bigl\}$ (see e.g. \cite[Appendix D]{CAS}), where
\begin{equation}\label{UERF}
\phi_{\omega}(x)=e^{\frac{\omega+1}{2}}e^{-\frac{1}{2}x^{2}}.
\end{equation}
In addition, $\phi_{\omega}(x)$ is the only minimizer (modulo translation and phase) of the problem
\begin{align}
\begin{split}\label{MPEaa}
d(\omega)&={\inf}\left\{S_{\omega}(u):\, u\in H^{1}(\mathbb{R})\cap L^{A}(\mathbb{R}) \setminus  \left\{0 \right\},  I_{\omega}(u)=0\right\} \\
&=\frac{1}{2}\,{\inf}\bigl\{\left\|u\right\|_{L^{2}}^{2}:u\in  H^{1}(\mathbb{R})\cap L^{A}(\mathbb{R}) \setminus \left\{0 \right\},  I_{\omega}(u)\leq 0 \bigl\},
\end{split}
\end{align}
where
\begin{align*}
 S_{\omega}(u)&=\frac{1}{2} \|\partial^{}_{x}u \|^{2}_{L^{2}}+ \frac{\omega+1}{2}\|u \|^{2}_{L^{2}}-\frac{1}{2}\int_{\mathbb{R}}\left|u\right|^{2}\mbox{Log}\left|u\right|^{2}dx,\\
 I_{\omega}(u)&= \|\partial^{}_{x}u \|^{2}_{L^{2}}+ \omega\,\|u \|^{2}_{L^{2}}-\int_{\mathbb{R}}\left|u\right|^{2}\mbox{Log}\left|u\right|^{2}dx.\end{align*}
Moreover $d(\omega)=e^{\omega+1}\sqrt{\pi}/2$.  For the proof of this result we refer to A. H. Ardila \cite{AHA1} (see also Cazenave and Lions \cite[Remark II.3]{CALO}).

\begin{lemma}\label{L5}
The set of the minimizers for the problem
\begin{equation*}
d^{0}(\omega)= \mathrm{inf}\bigl\{S_{\omega}(u):u\in W(\dot{\mathbb{R}})  \verb'\'  \left\{0 \right\},  I_{\omega}(u)=0\bigl\},
\end{equation*}
is given by $\left\{e^{i\theta}\chi_{+}\phi_{\omega},\,e^{i\theta}\chi_{-}\phi_{\omega}:\, \theta\in\mathbb{R} \right\}$, where $\phi_{\omega}$ is defined in \eqref{UERF}.
\end{lemma}
\begin{proof}
We use the argument in \cite[Lemma 4]{ADNP}. First, we remark that the following variational problem is equivalent to $d^{0}(\omega)$:
\begin{equation}\label{3LS}
 d^{1}(\omega)=  \frac{1}{2}\mbox{\rm inf} \bigl\{\left\|u\right\|_{L^{2}}^{2} :u\in W(\dot{\mathbb{R}})  \verb'\'  \left\{0 \right\},  I_{\omega}(u)\leq 0\bigl\}.
\end{equation}
Arguing as in Lemma \ref{L2} we can show that the quantity $d^{1}(\omega)$ is positive. Now, let $u\in H_{rad}^{1}(\mathbb{R})\cap L^{A}(\mathbb{R})$ be such that $I_{\omega}(\chi_{+}u)\leq 0$. Then, since $u$ is even, we have $I_{\omega}(u)=2\,I_{\omega}(\chi_{+}u)\leq0$. Thus, from \eqref{MPEaa}, we see that
\begin{equation}\label{3LSW}
 \frac{1}{2}\left\|\chi_{+}u\right\|^{2}_{L^{2}}=\frac{1}{4}\left\|u\right\|^{2}_{L^{2}}\geq \frac{1}{4}\left\|\phi_{\omega}\right\|^{2}_{L^{2}}=\frac{1}{2}\left\|\chi_{+}\phi_{\omega}\right\|^{2}_{L^{2}},
\end{equation}
and  $I_{\omega}(\chi_{+}\phi_{\omega})=0$. Therefore, $\chi_{+}\phi_{\omega}$ is a minimizer of $L^{2}(\mathbb{R})$-norm among the functions of $W(\dot{\mathbb{R}})$, supported on $\mathbb{R}^{+}$ and satisfying $I_{\omega}\leq 0$. We observe that the equality in \eqref{3LSW} is satisfied  if and only if $u(x)=\phi_{\omega}(x)$ for all $x\in \mathbb{R}$ (modulo  phase).
Indeed, suppose we have the equality in \eqref{3LSW}. Since $u$ satisfies  $I_{\omega}(u)\leq0$, we have
\begin{equation*}
S_{\omega}(u)=\frac{1}{2}I_{\omega}(u)+\frac{1}{2}\left\|u\right\|^{2}_{L^{2}}\leq \frac{1}{2}\left\|u\right\|^{2}_{L^{2}}=\frac{1}{2}\left\|\phi_{\omega}\right\|^{2}_{L^{2}}=S_{\omega}(\phi_{\omega}),
\end{equation*}
and thus $u(x)=\phi_{\omega}(x-y)$ for some $y\in \mathbb{R}$. Moreover, since $u$ is even, we have that $y=0$.

Next, we recall that for any function $u\in W(\dot{\mathbb{R}})$ there exists a unique couple of functions $u_{+}$, $u_{-}\in H_{rad}^{1}(\mathbb{R})\cap L^{A}(\mathbb{R})$ such that $u=\chi_{+} u_{+}+ \chi_{-} u_{-}$. Thus, if $I_{\omega}(u)\leq0$, then one of the following alternative holds: $I_{\omega}(\chi_{+}u)\leq0$  or $I_{\omega}(\chi_{-}u)\leq0$. Furthermore, using \eqref{3LSW} we obtain
\begin{equation}\label{UIL}
 \frac{1}{2}\left\|\chi_{+} u_{+}\right\|^{2}_{L^{2}}\geq \frac{1}{2}\left\|\chi_{+}\phi_{\omega}\right\|^{2}_{L^{2}} \,\,\,\, \text{or}\,\,\,\, \frac{1}{2}\left\|\chi_{-} u_{-}\right\|^{2}_{L^{2}}\geq \frac{1}{2}\left\|\chi_{-}\phi_{\omega}\right\|^{2}_{L^{2}}.
\end{equation}
In addition, the equality in \eqref{UIL} is satisfied  if and only if $\left|u_{+}\right|=\phi_{\omega}$ or $\left|u_{-}\right|=\phi_{\omega}$, which concludes the proof.
\end{proof}

\begin{lemma} \label{L34}
Let $\gamma>0$. The following  inequality holds for any $\omega\in \mathbb{R}$:
\begin{equation}\label{EIN}
 d_{\gamma}(\omega)<d^{0}(\omega).
\end{equation}
\end{lemma}
\begin{proof}
Since we have
\begin{equation*}
{I}_{\omega,\gamma}\left(\chi_{+}\phi_{\omega}\right)=I_{\omega}\left(\chi_{+}\phi_{\omega}\right)-\gamma^{-1}e^{\omega+1}=-\gamma^{-1}e^{\omega+1}<0,
\end{equation*}
we infer that there exist $\lambda\in\left(0,1\right)$ such that  ${I}_{\omega,\gamma}\left(\lambda\chi_{+}\phi_{\omega}\right)=0$. Thus, from Lemma \ref{L5} and by the definition of ${d}_{\gamma}(\omega)$ we have
\begin{equation*}
{d}_{\gamma}(\omega)\leq {S}_{\omega,\gamma}\left(\lambda\chi_{+}\phi_{\omega}\right)=\frac{\lambda^{2}}{2}\left\|\chi_{+}\phi_{\omega}\right\|^{2}_{L^{2}(\mathbb{R})}<\frac{1}{2}\left\|\chi_{+}\phi_{\omega}\right\|^{2}_{L^{2}(\mathbb{R})}=d^{0}(\omega),
\end{equation*}
and the lemma is proved.
\end{proof}

The proof of the following lemma can be found in \cite[Lemma 4.10]{AnguloArdila2016}, and is presented here for the sake of completeness.

\begin{lemma} \label{L4}
Let  $\left\{u_{n}\right\}$ be a bounded sequence in $W(\dot{\mathbb{R}})$ such that $u_{n}\rightarrow u$ a.e. in $\mathbb{R}$. Then $u\in W(\dot{\mathbb{R}})$ and
\begin{equation*}
\lim_{n\rightarrow \infty}\int_{\mathbb{R}}\left\{\left|u_{n}\right|^{2}\mathrm{Log}\left|u_{n}\right|^{2}dx-\left|u_{n}-u\right|^{2}\mathrm{Log}\left|u_{n}-u\right|^{2}\right\}dx=\int_{\mathbb{R}}\left|u\right|^{2}\mathrm{Log}\left|u\right|^{2}dx.
\end{equation*}
\end{lemma}
\begin{proof}
We first recall that, by \eqref{IFD}, $\left|z\right|^{2}\mbox{Log}\left|z\right|^{2}=A(\left|z\right|)-B(\left|z\right|)$  for every  $z\in\mathbb{C}$.  By the weak-lower semicontinuity of the ${L^{2}(\mathbb{R})}$-norm and Fatou lemma we have $u\in W(\dot{\mathbb{R}})$.  It is clear that the sequence $\left\{u_{n}\right\}$ is  bounded in ${L^{A}(\mathbb{R})}$. Since $A$ is convex and increasing function with $A(0)=0$, it is follows from Br\'ezis-Lieb lemma \cite[Theorem 2 and Examples (b)]{LBL} that
\begin{equation}\label{AC}
\lim_{n\rightarrow \infty}\int_{\mathbb{R}} \left| A(\left|u_{n}\right|)-A(\left|u_{n}-u\right|)- A(\left|u_{}\right|)\right|dx=0.
\end{equation}
On the other hand, thanks to the continuous embedding $W(\dot{\mathbb{R}})\hookrightarrow \Sigma$, we have that the sequence $\left\{u_{n}\right\}$ is also bounded in $\Sigma$. An easy computations shows that the function $B$ is convex, increasing and  nonnegative with
$B(0)=0$. Furthermore, by H\"{o}lder and Sobolev inequalities, for any $u$, $v\in \Sigma$  we have that (see \cite[Lemma 1.1]{CL})
\begin{equation}\label{DB}
\int_{\mathbb{R}}\left|B(\left|u(x)\right|)- B(\left|v(x)\right|)\right|dx\leq C\left(1+ \left\|u\right\|^{2}_{\Sigma}+ \left\|v\right\|^{2}_{\Sigma} \right)\left\|u-v\right\|_{{L^{2}}}.
\end{equation}
Then, the function $B$ satisfies the hypotheses of \cite[Theorem 2 and Examples (b)]{LBL} and therefore
\begin{equation}\label{BC}
\lim_{n\rightarrow \infty}\int_{\mathbb{R}} \left| B(\left|u_{n}\right|)-B(\left|u_{n}-u\right|)- B(\left|u_{}\right|)\right|dx=0.
\end{equation}
Thus the result follows from \eqref{AC} and  \eqref{BC}.
\end{proof}

\begin{proof}[ \it {Proof of Theorem \ref{ESSW}}]
We use the argument in \cite[Theorem 4.4]{AnguloArdila2016}(see also \cite{ADNP,FO}). First, every minimizing sequence of \eqref{MPE} is bounded in $ W(\dot{\mathbb{R}})$. Let $\left\{ u_{n}\right\}$ be a minimizing sequence. We remark that the sequence $\left\{ u_{n}\right\}$ is bounded in $L^{2}(\mathbb{R})$. Now, by \eqref{EA2}, the logarithmic Sobolev inequality \eqref{LSA} and recalling that $I_{\omega,\gamma}(u_{n})=0$, we see for $\alpha>0$,
\begin{equation*}
\left(\frac{1}{2}-\frac{\alpha^{2}}{\pi}\right)\left\|u_{n}^{\prime}\right\|^{2}_{L^{2}}\leq \left(\mbox{Log}\left(\frac{e^{\frac{8}{\gamma^{2}}}e^{-\left(\omega+1\right)}}{\alpha^{}}\right)\right)\left\|u_{n}\right\|^{2}_{L^{2}} +\mbox{Log}\left(2\left\|u_{n}\right\|^{2}_{L^{2}}\right) \left\|u_{n}\right\|^{2}_{L^{2}}.
\end{equation*}
Taking $\alpha>0$ sufficiently small, we have that $\left\{ u_{n}\right\}$ is bounded in $\Sigma$. Moreover,  it follows from  $I_{\omega,\gamma}(u_{n})=0$ and \eqref{DB} that
\begin{equation*}
\int_{\mathbb{R}^{}}A\left(\left|u_{n}(x)\right|\right)dx\leq \int_{\mathbb{R}^{}}B\left(\left|u_{n}(x)\right|\right)dx+|\omega|\left\|u_{n}\right\|^{2}_{L^{2}}\leq C,
\end{equation*}
which implies, by \eqref{DA1} in  Appendix, that the sequence $\left\{ u_{n}\right\}$ is bounded in $W(\dot{\mathbb{R}})$. In addition, since $W^{}(\dot{\mathbb{R}})$ is a reflexive Banach space, there exist $\varphi \in W(\dot{\mathbb{R}^{}})$ such that, up to a subsequence, $u_{n}\rightharpoonup \varphi$ weakly in $W^{}(\dot{\mathbb{R}^{}})$ and $u_{n}(x)\rightarrow \varphi(x)$ $a.e.$  $x\in\mathbb{R}$.

Secondly, we show that $\varphi$ is nontrivial. We remark that the weak convergence in $W(\dot{\mathbb{R}^{}})$ implies that
\begin{equation}\label{3MNB}
\lim_{n\rightarrow \infty}u_{n}(0+)=\varphi(0+)\quad \text{and}\quad \lim_{n\rightarrow \infty}u_{n}(0-)=\varphi(0-).
\end{equation}
Indeed, from $W(\dot{\mathbb{R}^{}})\hookrightarrow \Sigma$  we have that $u_{n}\rightharpoonup \varphi$ weakly in $\Sigma$. Since in addition we have the  compact embedding $H^{1}(0,1)\hookrightarrow C[0,1]$, we get \eqref{3MNB}. Now, suppose that $\varphi\equiv 0$. Since $u_{n}$ satisfies ${I}_{\omega,\gamma}(u_{n})=0$, it follows from \eqref{3MNB} that
\begin{equation}\label{wert}
\lim_{n\rightarrow \infty} I_{\omega}(u_{n})=\lim_{n\rightarrow \infty}\left[{I}_{\omega,\gamma}(u_{n})+\gamma^{-1}\left|u_{n}(0+)-u_{n}(0-)\right|^{2}\right]=0.
\end{equation}
Define the sequence $v_{n}(x)=\lambda_{n}u_{n}(x)$ with
\begin{equation*}
\lambda_{n}=\exp\left(\frac{I_{\omega}(u_{n})}{2\|u_{n}\|^{2}_{L^{2}}}\right),
\end{equation*}
where $\exp(x)$ represent the exponential function. Then, it follows from  \eqref{wert} that $\lim_{n\rightarrow \infty}\lambda_{n}=1$. Moreover,  an easy calculation shows that $I_{\omega}(v_{n})=0$ for any $n\in \mathbb{N}$. Thus, by  the definition of $d^{0}(\omega)$ and Lemma \ref{L5} leads to
\begin{equation*}
d^{0}(\omega)\leq \lim_{n\rightarrow \infty} S_{\omega}(\lambda_{n}u_{n})=\lim_{n\rightarrow \infty}\,\lambda_{n}^{2}\,{S}_{\omega,\gamma}(u_{n})= {d}_{\gamma}(\omega).
\end{equation*}
that it is contrary to \eqref{EIN} and therefore we conclude that $\varphi$ is nontrivial.

Finally, we prove that  $I_{\omega,\gamma}(\varphi)=0$ and $\varphi\in \mathcal{N}_{\omega,\gamma}$.  If we suppose that $I_{\omega,\gamma}(\varphi)<0$, by elementary computations we find that  there is $\lambda\in (0,1)$ such that $I_{\omega,\gamma}(\lambda \varphi)=0$. Then, from the definition of $d_{\gamma}(\omega)$ and  the weak lower semicontinuity of the $L^{2}(\mathbb{R}^{})$-norm, we have
\begin{equation*}
d_{\gamma}(\omega)\leq \frac{1}{2}\left\|\lambda \varphi\right\|^{2}_{L^{2}}<\frac{1}{2}\left\|\varphi\right\|^{2}_{L^{2}}\leq \frac{1}{2}\liminf\limits_{n\rightarrow \infty}\left\|u_{n}\right\|^{2}_{L^{2}}=d_{\gamma}(\omega),
\end{equation*}
it which is impossible. Now suppose that ${I}_{\omega,\gamma}(\varphi)>0$. Since we have that $u_{n}\rightharpoonup \varphi$ weakly in $\Sigma$, from \eqref{3MNB} we get
\begin{eqnarray}\label{3C11}
\mathfrak{t_{\gamma}}[u_{n}]&-&\mathfrak{t_{\gamma}}[u_{n}-\varphi]-\mathfrak{t_{\gamma}}[\varphi]\rightarrow0 \\ \label{3C12}
\left\|u_{n}\right\|^{2}_{L^{2}}&-&\left\|u_{n}-\varphi\right\|^{2}_{L^{2}}-\left\|\varphi\right\|^{2}_{L^{2}}\rightarrow0,
\end{eqnarray}
as $n\rightarrow\infty$. Combining \eqref{3C11}, \eqref{3C12} and Lemma \ref{L4} we obtain
\begin{equation*}
\lim_{n\rightarrow \infty}I_{\omega, \gamma}(u_{n}-\varphi)=\lim_{n\rightarrow \infty}I_{\omega, \gamma}(u_{n})-I_{\omega,\gamma}(\varphi)=-I_{\omega,\gamma}(\varphi),
\end{equation*}
which combined with  $I_{\omega, \gamma}(\varphi)> 0$ give us  that $I_{\omega, \gamma}(u_{n}-\varphi)<0$ for sufficiently large $n$. Thus, by \eqref{3C12} and  applying the same argument as above, we see that
\begin{equation*}
d_{\gamma}(\omega)\leq\frac{1}{2} \lim_{n\rightarrow \infty}\left\|u_{n}-\varphi\right\|^{2}_{L^{2}}=d_{\gamma}(\omega)-\frac{1}{2}\left\|\varphi\right\|^{2}_{L^{2}},
\end{equation*}
which is a contradiction because $\|\varphi\|^{2}_{L^{2}}>0$. Then, we deduce that $I_{\omega,\gamma}(\varphi)=0$. This proves the first part of the statement.  Furthermore, by the weak lower semicontinuity of the $L^{2}(\mathbb{R}^{})$-norm,  we have
\begin{equation*}
d_{\gamma}(\omega)\leq \frac{1}{2}\left\|\varphi\right\|^{2}_{L^{2}}\leq \frac{1}{2} \liminf\limits_{n\rightarrow \infty}\left\|u_{n}\right\|^{2}_{L^{2}}=d_{\gamma}(\omega),
\end{equation*}
which implies, by the definition of $d_{\gamma}(\omega)$, that $\varphi\in \mathcal{N}_{\omega,\gamma}$.
\end{proof}

\section{Characterizations of ground states}
\label{S:3}
The aim of this section is to prove Theorem \ref{12}. Some preparation is needed.
\begin{lemma} \label{L8}
Let  $u\in C^{2}(\mathbb{R}^{+})\cap H^{1}(\mathbb{R}^{+})$ be a non-trivial solution of
\begin{equation}\label{A22}
-u^{\prime \prime}+\omega u-u\,  \mathrm{Log}\left|u\right|^{2}=0, \quad \text{ on} \quad \mathbb{R}^{+}.
\end{equation}
Then there exist $\theta_{}\in \mathbb{R}$ and $c\in\mathbb{R}$  such that
\begin{equation}\label{CLO}
u(x)=e^{i\theta_{} }e^{\frac{\omega+1}{2}}e^{-\frac{1}{2}(x+c)^{2}},\quad  \text{ for all $x\in\mathbb{R}^{+}$}.
\end{equation}
The same conclusion can be reached if we replace $\mathbb{R}^{+}$ by  $\mathbb{R}^{-}$.
\end{lemma}
\begin{proof}
We  may write  $u(x)=e^{i\theta(x)}\rho (x)$, where $\theta$, $\rho\in C^{2}(\mathbb{R}^{+})$ and $\rho\geq0$. Multiplying the equation \eqref{A22} by $\overline{u^{\prime}}$, we obtain
\begin{equation}\label{EQQ123}
\left|u^{\prime}(x)\right|^{2}-\left(\omega+1\right)\left|u(x)\right|^{2}+\left|u(x)\right|^{2}\mathrm{Log}\left|u(x)\right|^{2}\equiv K,
\end{equation}
where $K\in \mathbb{R}$. Since $u\in H^{1}(\mathbb{R}^{+})$, it follows that $u(x)\rightarrow 0$ as $x\rightarrow \infty$. Thus, by \eqref{A22}, $u^{\prime\prime}(x)\rightarrow 0$ as $x\rightarrow \infty$, and so $u^{\prime}(x)\rightarrow 0$ as $x\rightarrow \infty$. Then, letting  $x\rightarrow \infty$  in \eqref{EQQ123}, we see that $K=0$, so
\begin{equation}\label{EQQ0}
\left|u^{\prime}(x)\right|^{2}-\left(\omega+1\right)\left|u(x)\right|^{2}+\left|u(x)\right|^{2}\mathrm{Log}\left|u(x)\right|^{2}\equiv 0.
\end{equation}

Next, writing  the system of equations satisfied by $\theta$ and $\rho$ we have in particular that $\rho\,\theta^{\prime\prime}+ 2\rho^{\prime}\,\theta^{\prime}\equiv0$. Which implies that there exist $\hat{K}\in \mathbb{R}$ such that $\rho^{2}\theta^{\prime}\equiv \hat{K}$. On the other hand, by \eqref{EQQ0} we have that $\left|u^{\prime}\right|$ is bounded, it follows that $\rho^{2}(\theta^{\prime})^{2}$ is bounded. Since $\rho(x)\rightarrow 0$ as $x\rightarrow \infty$, we must have $\hat{K}\equiv0$. Therefore, $\theta^{\prime}\equiv 0$ and  $u(x)=e^{i\theta}\rho (x)$, where $\theta\in \mathbb{R}$ and  $\rho\geq0$ satisfies
\begin{equation}\label{A33}
-\rho^{\prime \prime}+\omega \rho-\rho\,  \mathrm{Log}\left|\rho\right|^{2}=0, \quad  \text{ on} \quad \mathbb{R}^{+}.
\end{equation}
We  observe that if we take $\beta(s)=\omega s-s\,\mathrm{Log}\,s^{2}$, since $\beta\in C[0,+\infty)$, is nondecreasing for $s$ small,  $\beta(0) = 0$ and $\beta(\sqrt{e^{\omega}})=0$, by \cite[Theorem 1]{LVL} we have that each solution $\rho\geq 0$ is either trivial or strictly positive.
Since $u\neq 0$, we infer that $\rho>0$ on $\mathbb{R}^{+}$. Finally, the equation \eqref{A33} may be integrated using standard arguments. Indeed, by explicit integration there exist $c\in \mathbb{R}$ such that for $x>0$ (see \cite{CMKA}),
 \begin{equation*}
\rho(x)=e^{\frac{\omega+1}{2}}e^{-\frac{1}{2}(x+c)^{2}},
\end{equation*}
which completes of proof.
\end{proof}
\begin{lemma} \label{LCU}
Let $\gamma\in\mathbb{R}\setminus \left\{0\right\}$,  $\omega\in \mathbb{R}$ and $v\in W(\dot{\mathbb{R}})$ be a solution of \eqref{GS}. Then, $v$ verifies the following:
\begin{align}
& v\in C^{j}(\mathbb{R}\setminus \left\{0 \right\}), \quad j=1, 2, \label{1} \\
& -v^{\prime\prime}+\omega v-v\,  \mathrm{Log}\left|v \right|^{2}=0 \quad \mbox{on}\quad \mathbb{R}\setminus \left\{0 \right\},\label{2}  \\
& v^{\prime}(0+)=v^{\prime}(0-)\quad \mbox{and}\quad  v^{}(0+)-v^{}(0-)=- \gamma v^{\prime}(0), \label{3}\\
&  v^{\prime}(x), v(x)\rightarrow 0, \quad  \mbox{as} \quad \left|x\right|\rightarrow\infty. \label{4}
\end{align}
\end{lemma}
\begin{proof}
The proof of item \eqref{1} follow by a standard bootstrap argument using test functions $\xi \in C^{\infty}_{0}(\mathbb{R} \setminus \left\{0 \right\})$ (see e.g. \cite[Chapter 8]{CB}). Indeed, from \eqref{GS} applied with $\xi\,v $  we deduce that
\begin{equation}\label{3ALU}
-(\xi v)^{\prime\prime}+\omega \xi v=-\xi^{\prime\prime} v-2 \xi^{\prime} v^{\prime}+\xi v\, \mbox{Log}\left|v\right|^{2},
\end{equation}
in the sense of distributions on $\mathbb{R}\setminus \left\{0 \right\}$. The right hand side is in ${L^{2}(\mathbb{R})}$  and so $\xi v\in H^{2}(\mathbb{R})$. This implies that $v$ is in $C^{2}(\mathbb{R}\setminus \left\{0 \right\})\cap H_{\rm loc}^{2}(\mathbb{R})$ and is a classical solution of this equation on $\mathbb{R}\setminus \left\{0 \right\}$,  from which \eqref{1} and \eqref{2} follows.

From the characterization given  by Lemma \ref{L8}, $v$ is the form \eqref{CLO} on each side of the origin. Set $w= -\omega v+v\,\mathrm{Log}\left|v\right|^{2}$. It is clear that $w\in L^{2}(\mathbb{R})$. Moreover, by \eqref{GS}, we see that
\begin{equation*}
\mathfrak{t_{\gamma}}[v,z]=\left(w,z\right)\quad \textrm{for all} \quad z\in W(\dot{\mathbb{R}}).
\end{equation*}
Thus, from the fact that $ W(\dot{\mathbb{R}})$ is dense in $\Sigma$, we have
\begin{equation}\label{YUT}
\mathfrak{t_{\gamma}}[v,z]=\left(w,z\right)\quad \textrm{for all} \quad z\in \Sigma.
\end{equation}
Now, we recall that the form $\mathfrak{t_{\gamma}}$  is associated with the operator  self-adjoint $H_{\gamma}$. Then the theory of representation of forms by operators \cite[Theorem 10.7]{KSN} implies that $v\in \mathrm{dom}({H}_{\gamma})$; that is, $v$ satisfies the two conditions in \eqref{3}. Finally, the proof of \eqref{4} is contained in Lemma \ref{L8}. This concludes the proof.
\end{proof}

\begin{proof}[ \it {Proof of Proposition \ref{11}}]
Let $u\in \mathcal{N}_{\omega,\gamma}$. By Remark \ref{RM}, we see that $u$ satisfies the stationary problem  \eqref{GS}. From Lemma \ref{LCU} and the characterization give by Lemma \ref{L8}, we see that all possible solutions to \eqref{GS} must be given by
\begin{equation}\label{LKI}
v(x)=
\begin{cases}
e^{\theta_{1}}e^{\frac{\omega+1}{2}}e^{-\frac{1}{2}(x+t_{1})^{2}}, &\text{if $x>0$;}\\
e^{\theta_{2}}e^{\frac{\omega+1}{2}}e^{-\frac{1}{2}(x-t_{2})^{2}}, &\text{if $ x<0$;}
\end{cases}
\end{equation}
where $\theta_{1}$, $\theta_{2}\in\mathbb{R}$ and the couple $(t_{1},t_{2})\in \mathbb{R}^{}\times\mathbb{R}^{}$. Notice that, by Lemma \ref{LCU}, the solution $v$ must satisfy the boundary conditions \eqref{3}.

Now, since $u$ is a minimizer of ${S}_{\omega,\gamma}$ under the constraint $I_{\omega,\gamma}(u)=0$, we have that $e^{i\theta_{1}}= -e^{i\theta_{2}}$. Indeed, once fixed $t_{1}$ and $t_{2}$, it is clear that such condition minimizes the quadratic form $\mathfrak{t_{\gamma}}$, while the other terms in the functional ${S}_{\omega,\gamma}$ are the same.  This explains the negative sign in \eqref{Qo}. Taking into account the phase invariance of the problem we can choose $\theta_{1}=0$ and $\theta_{2}=\pi$.

Finally, from \eqref{3} and \eqref{LKI}, the boundary conditions for $v$ can be converted to the system \eqref{3S} for the unknowns $\theta_{1}$ and $\theta_{1}$. We remark that, by the first equation of the system \eqref{3S}, $t_{1}$ and $t_{1}$ must have the same sign. Thus, since $\gamma>0$, by the second equation of the system \eqref{3S}, we have that $t_{1}>0$ and $t_{2}>0$. This concludes the proof.
\end{proof}
In order to identify the ground states we  must find the solutions of system \eqref{3S}. As a first step in that direction we have the following lemma.

\begin{lemma} \label{3L}
For any $\gamma>2$, the function
\begin{equation*}
h(t)=\left(t+1\right)^{2}\left(1-\frac{1}{t^{2}}\right)-\gamma^{2}\mathrm{Log}\left(t^{2}\right),\end{equation*}
has exactly one zero in the interval $(1, \infty)$.
\end{lemma}
\begin{proof}
It is clear that $h\in C^{2}(\mathbb{R}^{+})$, $h(1)=0$ and $h(t)\rightarrow \infty$ as $t\rightarrow\infty$. By direct computations, we see that $h^{\prime}(1)=2(4-\gamma^{2})<0$ for $\gamma>2$. Then, there exist $x_{0}$,  $x_{1}\in (1, \infty)$ such that $h(x_{0})<0$ and $h(x_{1})>0$. Thus, the function $h$ must have at least one zero on the interval $(1, \infty)$. Moreover, we have that $h^{\prime \prime}(t)>0$ for $t\geq1$, which implies  that the function $h$ has exactly one zero.
\end{proof}

\begin{proposition} \label{3LL}
(i) Let $0<\gamma\leq 2$. Then the system \eqref{3S} has a unique solution given by $t_{1}={2}{\gamma}^{-1}$, $t_{2}= {2}{\gamma}^{-1}$.\\
(ii) Let $\gamma>2$. Then the system \eqref{3S} has three solutions;  one of them is given by $t_{1}={2}{\gamma}^{-1}$, $t_{2}= {2}{\gamma}^{-1}$.
\end{proposition}
\begin{proof} We use the argument in \cite[Theorem 5.3]{ADNP}.
We set $f(t)=te^{-\frac{1}{2}t^{2}}$ for all $t\geq 0$. It is clear that the first equation of the system \eqref{3S} is equivalent to $f(t_{1})=f(t_{2})$. Notice that $f(0)=0$, $f(t)>0$ for all $t>0$, and $f(t)\rightarrow 0$ as $t\rightarrow \infty$. Moreover, $f$ has a unique critical point at $t=1$,  which is a maximum.

It is not hard to see that the set of the solutions of the first equation in \eqref{3S}  with $t_{1}\leq t_{2}$ consists of the union of the following curves:
\begin{align*}
\mathcal{I}_{1}&:=\left\{(t_{1},t_{2})\in \mathbb{R}^{+}\times\mathbb{R}^{+}\,\,:\,\,t_{1}=t_{2} \right\}\\
\mathcal{I}_{2}&:=\left\{(t_{1},t_{2})\in(0,1)\times(1,\infty) \,\,:\,\,f(t_{1})=f(t_{2})\right\}.
\end{align*}
We remark that due to the regularity of $f$, $\mathcal{I}_{2}$ is a regular curve; the curve $\mathcal{I}_{1}\cup \mathcal{I}_{2}$ is given in Figure \ref{GUUN}.

Next, for each $\gamma$ positive the second equation of system \eqref{3S} is a hyperbola in the plane $(t_{1}, t_{2})$. Thus, the solution of the system of equations \eqref{3S} is the  intersection of these hyperbolas with $\mathcal{I}_{1}\cup \mathcal{I}_{2}$.  In order to find  all points of intersection of the two curves,  it is convenient to prove that
\begin{align}\label{3CNN}
\inf_{ t_{1},t_{2} \in \mathcal{I}_{1}}t^{-1}_{1}+t^{-1}_{2}&=0,\\\label{3CNP}
\inf_{ t_{1},t_{2} \in \mathcal{I}_{2}}t^{-1}_{1}+t^{-1}_{2}&=2.
\end{align}
To show this, we use the Lagrange multiplier method. We note that \eqref{3CNN}  is obvious.

Now, to find critical points of the function $t^{-1}_{1}+t^{-1}_{2}$  on the set $\mathcal{I}_{2}$, we must solve the following equation
\begin{gather}\label{3CNW}
\begin{aligned}
t^{2}_{1}f^{\prime}(t_{1})&=-t^{2}_{2}f^{\prime}(t_{2}), \quad \text{ $t_{1}\neq 1$ and  $t_{2}\neq 1$.} \\
\end{aligned}
\end{gather}
In particular, it follows from  $f(t_{1})=f(t_{2})$ and \eqref{3CNW} that $t_{1}-t^{3}_{1}=t^{3}_{2}-t_{2}$. Moreover, since $t_{1}$ satisfies $t_{1}\in(0,1)$, we have that $t_{2} \in(1,2\sqrt{3}/3)$. On the other hand, set $g(t)=t^{2}f^{\prime}(t)$. It is not hard to see that $g>0$ in $(0,1)$ and $g<0$ in $(1,\infty)$. Thus, the condition $g(t_{1})=-g(t_{2})$ with $0<t_{1}<1<t_{2}<+\infty$  is equivalent to $g^{2}(t_{1})=g^{2}(t_{2})$, $0<t_{1}<1<t_{2}<2\sqrt{3}/3$.

Since $g^{2}(1)=0$, we have
\begin{equation}\label{3OO1}
g^{2}(t)=\int^{t}_{1}\frac{d}{ds}g^{2}(s)ds=\int^{t}_{1}\left(4s^{3}\left(f^{\prime}(s)\right)^{2}+2s^{4}f^{\prime}(s)f^{\prime\prime}(s) \right)ds.
\end{equation}
We note that the function $f$ is increasing in the interval $(0,1)$, which implies that it is possible to perform the change of variable $t\rightarrow y=f(s)$ in the integral \eqref{3OO1} to obtain
\begin{equation*}
g^{2}(t_{1})=\int^{f(t_{1})}_{f(1)}p(\tau (y))dy,
\end{equation*}
where
\begin{equation}\label{3OO}
p(t):=4s^{3}f^{\prime}(s)+2s^{4}f^{\prime\prime}(s)=2t^{3}e^{-\frac{t^{2}}{2}}\left(t^{4}-5t^{2}+2\right),
\end{equation}
and the function $\tau$  is the inverse of $f$ in the interval $(0,1)$. Arguing in the same way on $(1,+\infty)$, we conclude that
\begin{equation}\label{3O1}
g^{2}(t_{2})=\int^{f(t_{2})}_{f(1)}p(\sigma(y))dy,
\end{equation}
where $\sigma$ is the inverse of $f$ in the interval $(1,+\infty)$. Now, in the interval $[1, 2\sqrt{3}/3]$ the function $p$ is negative and monotonically decreasing, which combined with \eqref{3O1} gives
\begin{equation}\label{3O3}
g^{2}(t_{2})=\int^{f(1)}_{f(t_{2})}\left|p(\sigma(y))\right|dy.
\end{equation}

Next, if $0<s_{1}<1<s_{2}<2\sqrt{3}/3$, then $\left|p(s_{1})\right|<\left|p(1)\right|<\left|p(s_{2})\right|$. Since in addition we have
$0<\tau(y)<1$, $1<\sigma(y)<2\sqrt{3}/3$ and $f(t_{1})=f(t_{2})$, we get
\begin{equation*}
g^{2}(t_{1})=\int^{f(t_{1})}_{f(1)}p(\tau (y))dy<\int^{f(1)}_{f(t_{2})}\left|p(\sigma(y))\right|dy=g^{2}(t_{2}),
\end{equation*}
where in the last identity we used \eqref{3O3}. This implies that $g^{2}(t_{1})=g^{2}(t_{2})$  can be obtained only when $t_{1}=t_{2}=1$.
Therefore, there are no critical point of $t^{-1}_{1}+t^{-1}_{2}$ on the set $\mathcal{I}_{2}$.  Thus, we see that  $t_{1}=t_{2}=1$ correspond to a minimum. This proves \eqref{3CNP}.

To conclude, combining \eqref{3CNN} and \eqref{3CNP} we infer that:

If $0<\gamma\leq2$, then the system \eqref{3S} has exactly one solution on the set $\mathcal{I}_{1}$ and is given by $t_{1}=t_{2}=2/\gamma$. Notice that, by \eqref{3CNP}, there are no solutions of the system \eqref{3S} on the set $\mathcal{I}_{2}$. Hence ${\rm (i)}$ follows.
\begin{figure}[htp]
\begin{center}
   \includegraphics[width=0.4\textwidth]{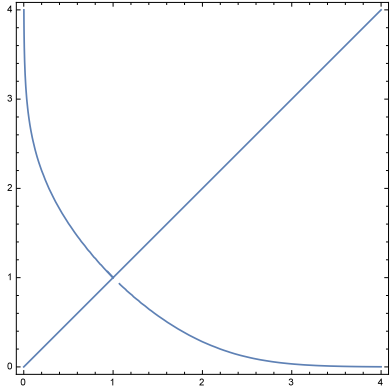}\\
  \caption{The graph of the curve $\mathcal{I}_{1}\cup \mathcal{I}_{2}$.} \label{GUUN}
  \end{center}
\end{figure}
If $\gamma>2$,  then the system \eqref{3S} has exactly three solutions: the first one lies in $\mathcal{I}_{1}$ and is given by $t_{1}=t_{2}=2/\gamma$. On the other hand, the first equation in \eqref{3S} implies that $t_{2}/t_{1}=\gamma t_{2}-1$. By substituting this expression into the second equation in \eqref{3S}, one finds the equation $h(t)=0$ with $t=t_{2}/t_{1}=\gamma t_{2}-1$; $h$ is defined in Lemma \ref{3L}.
We remark that  $t>1$. Indeed, $\gamma>2$ and $t_{2}>1$. Now, by Lemma \ref{3L}, we have that such equation has a unique solution $z_{0}$ in the interval $(1,+\infty)$. Thus, $\left((z_{0}+1)\gamma^{-1},z^{-1}_{0}(z_{0}+1)^{}\gamma^{-1}\right)$ is a solution to \eqref{3S}. Due to the symmetry of \eqref{3S} under change of $t_{1}$ and $t_{2}$, the third and last solution is given by $\left(z^{-1}_{0}(z_{0}+1)^{}\gamma^{-1},(z_{0}+1)\gamma^{-1}\right)$. This concludes the proof of ${\rm (ii)}$.
\end{proof}

\begin{proof}[ \it {Proof of Theorem \ref{12}}]
If $0<\gamma \leq 2$, by Propositions \ref{11} and \ref{3LL} we have $\mathcal{N}_{\omega,\gamma}\subset\bigl\{e^{i\theta}\phi^{t_{\ast},t_{\ast}}_{\omega}: \theta\in\mathbb{R} \bigl\}$, where $t_{\ast}=2/{\gamma}$. Moreover, by Theorem \ref{ESSW} we see that the set $\mathcal{N}_{\omega,\gamma}$ is not empty, which implies that $\mathcal{N}_{\omega,\gamma}=\bigl\{e^{i\theta}\phi^{t_{\ast},t_{\ast}}_{\omega}: \theta\in\mathbb{R} \bigl\}$. Thus we obtain the proof of (i) of Theorem \ref{12}.  Now we prove (ii) of Theorem \ref{12}.

If $\gamma>2$, by Propositions \ref{11} and \ref{3LL} we have $\mathcal{N}_{\omega,\gamma}\subset\bigl\{e^{i\theta}\phi^{t_{\ast},t_{\ast}}_{\omega}, e^{i\theta}\phi^{t_{1},t_{2}}_{\omega}, e^{i\theta}\phi^{t_{2},t_{1}}_{\omega}: \theta\in\mathbb{R} \bigl\}$, where $t_{\ast}=2{\gamma}^{-1}$ and the couple $(t_{1},t_{2})$ solves the system \eqref{3S} with $t_{1}<t_{2}$.  We note  that ${S}_{\omega,\gamma}\bigl(\phi^{t_{1},t_{2}}_{\omega}\bigl)={S}_{\omega,\gamma}\bigl(\phi^{t_{2},t_{1}}_{\omega}\bigl)$. Thus, in order to establish which stationary point is the minimizer we must compare ${S}_{\omega,\gamma}\bigl(\phi^{t_{1},t_{2}}_{\omega}\bigl)$ with ${S}_{\omega,\gamma}\bigl(\phi^{t_{\ast},t_{\ast}}_{\omega}\bigl)$.
Consider the functions
\begin{equation}\label{3IOP}
\Gamma(t)=\int^{\infty}_{t}e^{-s^{2}}ds, \quad t\in(0,\infty) \quad \mbox{and} \quad \sigma(t)=\frac{t}{\gamma\, t-1}, \quad
t\in (\gamma^{-1}, \infty).
\end{equation}
Next, we define the function
\begin{equation}\label{ANC}
n_{\gamma}(t)=\Gamma(t)+\Gamma(\sigma(t)), \quad \mbox{ for all $t\in (\gamma^{-1}, \infty)$.}
\end{equation}
It is clear that $n_{\gamma}\in C^{2}(\gamma^{-1},\infty)$.  Recalling that $\gamma^{-1}<t_{1}<t_{\ast}=2\gamma^{-1}<t_{2}$, we get by direct computations
\begin{equation}\label{3UNM}
{S}_{\omega,\gamma}\left(\phi^{t_{1},t_{2}}_{\omega}\right)=e^{\omega+1}\,n_{\gamma}(t_{1})\quad \mbox{and} \quad {S}_{\omega,\gamma}\left(\phi^{t_{\ast},t_{\ast}}_{\omega}\right)=e^{\omega+1}\,n_{\gamma}(2\gamma^{-1}).
\end{equation}
Therefore, we need to study the behaviour of the function  $n_{\gamma}$ in the interval $\left(\gamma^{-1},2\gamma^{-1}\right]$. We see that
its derivative is given by
\begin{equation*}
n^{\prime}_{\gamma}(t)=\frac{1}{\left(\gamma\,t-1\right)^{2}}e^{-\left(\sigma(t)\right)^{2}}-e^{-t^{2}}.
\end{equation*}
Now, since $t\in\left(\gamma^{-1},\infty\right)$,  the critical points of $n_{\gamma}$ are, by definition, the points where $n^{\prime}_{\gamma}(t)=0$. More precisely,
\begin{equation*}
\sigma(t)\,e^{-\frac{1}{2}\left(\sigma(t)\right)^{2}}=t\, e^{-\frac{1}{2}t^{2}},
\end{equation*}
which is equivalent to the first equation of \eqref{3S} in the unknowns $(t,\sigma(t))$. Moreover, by \eqref{3IOP}, we have
\begin{equation*}
\sigma(t)^{-1}+t^{-1}=\gamma,
\end{equation*}
which is equivalent to the second equation of \eqref{3S} in the unknowns $(t,\sigma(t))$. That is, the couple $(t,\sigma(t))$ solves the system
\eqref{3S}.  By Proposition \ref{3LL}, we have that for $\gamma >2$ there are three critical points for $n_{\gamma}$: $t_{1}$, $2\gamma^{-1}$ and  $t_{2}$. Moreover, by elementary computations, we see that
\begin{equation*}
n^{\prime\prime}_{\gamma}\left(2\gamma^{-1}\right)=2\gamma^{-1}e^{-\frac{4}{\gamma}}\left(4-\gamma^{2}\right),
\end{equation*}
which is negative if $\gamma >2$, it follows that the function $n_{\gamma}$ has a maximum value at $t_{\ast}=2\gamma^{-1}$. Thus, since $\gamma^{-1}<t_{1}<t_{\ast}<t_{2}$, we infer that $n_{\gamma}(t_{1})<n_{\gamma}(2\gamma^{-1})$. Moreover, using \eqref{3UNM}, we obtain ${S}_{\omega,\gamma}(\phi^{t_{1},t_{2}}_{\omega})<{S}_{\omega,\gamma}(\phi^{t_{\ast},t_{\ast}}_{\omega})$. Therefore, it  implies  that $\phi^{t_{1},t_{2}}_{\omega}$ and $\phi^{t_{2},t_{1}}_{\omega}$ are the minimizers of $d_{\gamma}(\omega)$. Thus we get the proof of (ii) of Theorem \ref{12}.
\end{proof}

\section{Stability of the ground states}
\label{S:4}
The proof of Theorem \ref{EST} relies on the following compactness result.

\begin{lemma} \label{CSM}
Let $\left\{ u_{n}\right\}\subseteq W(\dot{\mathbb{R}})$ be a minimizing sequence for $d_{\gamma}(\omega)$. Then, up to a subsequence,  there exist  $\varphi\in\mathcal{N}_{\omega,\gamma}$ such that $u_{n}\rightarrow \varphi$ strongly in $W(\dot{\mathbb{R}})$.
\end{lemma}
\begin{proof}
By Theorem \ref{ESSW}, we see that exist  $\varphi\in \mathcal{N}_{\omega,\gamma}$ such that, up to a subsequence, $u_{n}\rightharpoonup \varphi$ weakly  in $W(\dot{\mathbb{R}})$ and $u_{n}\rightarrow \varphi$ $a.e.$ in  $\mathbb{R}$. Notice that (see \eqref{3MNB})
\begin{equation}\label{34MNB}
\lim_{n\rightarrow \infty}u_{n}(0+)=\varphi(0+)\quad \text{and}\quad \lim_{n\rightarrow \infty}u_{n}(0-)=\varphi(0-).
\end{equation}
Furthermore, tanks to \eqref{3C12},
we have  $u_{n}\rightarrow \varphi$   in $L^{2}(\mathbb{R}^{})$. Then, since the sequence $\left\{ u_{n}\right\}$ is bounded in $\Sigma$, from \eqref{DB} we obtain
\begin{equation*}
 \lim_{n\rightarrow \infty}\int_{\mathbb{R}^{}}B\left(\left|u_{n}(x)\right|\right)dx=\int_{\mathbb{R}^{}}B\left(\left|\varphi(x)\right|\right)dx.
\end{equation*}
Which combined with $I_{\omega,\gamma}(u_{n})=I_{\omega,\gamma}(\varphi)=0$  for any $n\in \mathbb{N}$,  gives
\begin{equation}\label{2BX1}
\lim_{n\rightarrow \infty}\left[\mathfrak{t_{\gamma}}[u_{n}]+\int_{\mathbb{R}^{}}A\left(\left|u_{n}(x)\right|\right)dx\right]=\mathfrak{t_{\gamma}}[\varphi]+\int_{\mathbb{R}^{}}A\left(\left|\varphi(x)\right|\right)dx.
\end{equation}
Moreover, by \eqref{2BX1},  the weak lower semicontinuity of the $L^{2}(\mathbb{R}^{})$-norm and Fatou lemma, we deduce (see, for example, \cite[Lemma 12 in chapter V]{AH})
\begin{align}
& \lim_{n \to \infty}\mathfrak{t_{\gamma}}[u_{n}]=\mathfrak{t_{\gamma}}[\varphi],\label{N1}\\
& \lim_{n \to \infty}\int_{\mathbb{R}}A\left(\left|u_{n}(x)\right|\right)dx=\int_{\mathbb{R}}A\left(\left|\varphi(x)\right|\right)dx. \label{N2} \end{align}
Since $u_{n}\rightharpoonup \varphi$ weakly  in $\Sigma$, it follows from \eqref{34MNB} and \eqref{N1}  that $u_{n}\rightarrow\varphi$  in $\Sigma$.  Finally, by Proposition  \ref{orlicz}-{ii)} (Appendix below) and \eqref{N2} we have $u_{n}\rightarrow\varphi$  in $L^{A}(\mathbb{R})$. Thus, by definition of the $W(\dot{\mathbb{R}})$-norm, we infer that $u_{n}\rightarrow\varphi$  in $W(\dot{\mathbb{R}})$. Which concludes the proof.
\end{proof}

\begin{proof}[ \it {Proof of Theorem \ref{EST}}] Our proof is inspired by the results of \cite{NANAA, AnguloArdila2016}. The proof of part (i) in theorem, the stability of the ground state $\phi^{t_{\ast},t_{\ast}}_{\omega}$ for $0<\gamma \leq2$, follows along the same lines  as \cite[Theorem 1.2]{AnguloArdila2016}. We omit the details.

Next we prove (ii) of theorem. Fix $\gamma>2$. Now arguing by contradiction and suppose that $e^{i\omega t}\phi^{t_{1},t_{2}}_{\omega}$ is not stable in $W(\dot{\mathbb{R}})$, then there exist $\epsilon>0$, a sequence $(u_{n,0})_{n\in \mathbb{N}}$ such that
\begin{equation}\label{3C1}
\left\|u_{n,0}-\phi^{t_{1},t_{2}}_{\omega}\right\|_{ W(\dot{\mathbb{R}})}<\frac{1}{n},
\end{equation}
and a sequence $(\tau_{n})_{n\in \mathbb{N}}$ such that
\begin{equation}\label{3C2}
{\rm\inf\limits_{\theta\in \mathbb{R}}} \|u_{n}(\tau_{n})-e^{i\theta}\phi^{t_{1},t_{2}}_{\omega}\|_{W(\dot{\mathbb{R}})}=\frac{\epsilon}{2},
\end{equation}
where $u_{n}$ denotes the solution of the Cauchy problem \eqref{00NL} with initial data $u_{n,0}$.

With no loss of generality, we assume
\begin{equation}\label{3EX1}
\epsilon_{}\leq \inf_{\theta\in \mathbb{R}}\|\phi^{t_{1},t_{2}}_{\omega}-e^{i\theta}\phi^{t_{2},t_{1}}_{\omega}\|_{W(\dot{\mathbb{R}})}=\|\phi^{t_{1},t_{2}}_{\omega}-\phi^{t_{2},t_{1}}_{\omega}\|_{W(\dot{\mathbb{R}})}.
\end{equation}

Set $v_{n}(x)= u_{n}(x,\tau_{n})$. By \eqref{3C1} and conservation laws, as $n\rightarrow \infty$,
\begin{gather}
\left\|v_{n}\right\|^{2}_{L^{2}}=\left\|u_{n}(\tau_{n})\right\|^{2}_{L^{2}}=\left\|u_{n,0}\right\|^{2}_{L^{2}}\rightarrow \left\|\phi_{\omega,\gamma}\right\|^{2}_{L^{2}}\label{CE1} \\
E(v_{n})=E(u_{n}(\tau_{n}))=E(u_{n,0})\rightarrow E(\phi_{\omega,\gamma}).\label{CE2}
\end{gather}
In particular, it follows from \eqref{CE1} and \eqref{CE2} that, as $n\rightarrow \infty$,
\begin{equation}\label{A12}
S_{\omega,\gamma}(v_{n})\rightarrow S_{\omega,\gamma}(\phi^{t_{1},t_{2}}_{\omega})=d_{\gamma}(\omega).
\end{equation}
Moreover, combining \eqref{CE1} and \eqref{A12} leads to $I_{\omega,\gamma}(v_{n})\rightarrow 0$ as $n\rightarrow \infty$.
Define the sequence $f_{n}(x)=\rho_{n}v_{n}(x)$ with
\begin{equation*}
\rho_{n}=\exp\left(\frac{I_{\omega,\gamma}(v_{n})}{2\|v_{n}\|^{2}_{L^{2}}}\right),
\end{equation*}
where $\exp(x)$ is the exponential function. It is clear that $\lim_{n\rightarrow \infty}\rho_{n}=1$ and $I_{\omega,\gamma}(f_{n})=0$ for any $n\in\mathbb{N}$. Furthermore, since the sequence $\left\{v_{n}\right\}$  is bounded in $W(\dot{\mathbb{R}})$, we get
\begin{equation}\label{NUW}
\|v_{n}-f_{n}\|_{W(\dot{\mathbb{R}})}\rightarrow 0, \quad \mbox{as} \quad n\rightarrow \infty.
\end{equation}
Then, thanks to \eqref{A12}, we have that $\left\{f_{n}\right\}$ is a minimizing sequence for $d_{\gamma}(\omega)$. Therefore, by Lemma \ref{CSM}, up to a subsequence, there exist $\varphi\in\mathcal{N}_{\omega,\gamma}$  such that
\begin{equation}\label{WQE}
\|f_{n}-\varphi\|_{W(\dot{\mathbb{R}})}\rightarrow 0, \quad \mbox{as} \quad n\rightarrow \infty.
\end{equation}
Moreover, by Theorem \ref{12}, we see that either $\varphi=e^{i\theta}\phi^{t_{1},t_{2}}_{\omega}$ or $\varphi=e^{i\theta}\phi^{t_{2},t_{1}}_{\omega}$ for some value $\theta\in \mathbb{R}$.

Suppose that there exist $\theta\in \mathbb{R}$ such that $\varphi=e^{i\theta}\phi^{t_{1},t_{2}}_{\omega}$. It follows from \eqref{WQE}, that $f_{n}\rightarrow e^{i\theta}\phi^{t_{1},t_{2}}_{\omega}$ strongly in $W(\dot{\mathbb{R}})$. Then, by \eqref{NUW}, we have
\begin{align*}
\|u_{n}(\tau_{n})-e^{i\theta}\phi^{t_{1},t_{2}}_{\omega}\|_{ W(\dot{\mathbb{R}})}&\leq \|v_{n}-f_{n}\|_{W(\dot{\mathbb{R})}}+\|f_{n}-e^{i\theta}\phi^{t_{1},t_{2}}_{\omega}\|_{W(\dot{\mathbb{R}})}\rightarrow 0
\end{align*}
as $n\rightarrow+\infty$, which is a contradiction with \eqref{3C2} and thus the assumption of orbital instability of the standing wave $e^{i\omega t}\phi^{t_{1},t_{2}}_{\omega}$ proves false.

Now, suppose that there exist $\theta\in \mathbb{R}$ such that $\varphi=e^{i\theta}\phi^{t_{2},t_{1}}_{\omega}$. Combining \eqref{3C2} and \eqref{NUW} leads to
\begin{align*}
{\rm\inf\limits_{\theta\in \mathbb{R}}}\|f_{n}-e^{i\theta}\phi^{t_{1},t_{2}}_{\omega}\|_{ W(\dot{\mathbb{R}})}&\leq \frac{3}{5}\epsilon.
\end{align*}
for $n$ sufficiently large. Thus, up to a subsequence, there exist  $\left\{\theta_{n}\right\}\subseteq\mathbb{R}$ such that
\begin{equation}\label{3QRR}
\|f_{n}-e^{i\theta_{n}}\phi^{t_{1},t_{2}}_{\omega}\|_{ W(\mathbb{R}_{\circ})}\leq \frac{2}{3}\epsilon,
\end{equation}
which combined with triangular identity, \eqref{3EX1} and \eqref{3QRR} gives
\begin{align*}
\|f_{n}-e^{i\theta_{}}\phi^{t_{2},t_{1}}_{\omega}\|_{ W(\dot{\mathbb{R}})}&\geq \|e^{i\theta_{n}}\phi^{t_{1},t_{2}}_{\omega}-e^{i\theta}\phi^{t_{2},t_{1}}_{\omega}\|_{ W(\dot{\mathbb{R}})}-\|e^{i\theta_{n}}\phi^{t_{1},t_{2}}_{\omega}-f_{n}\|_{ W(\dot{\mathbb{R}})}\\
&\geq \frac{1}{3}\epsilon.
\end{align*}
This is a contradiction with \eqref{WQE}. Thus we get the proof of (ii) of Theorem \ref{EST}.
\end{proof}

\section*{Acknowledgments}
The author would like to thank the reviewer for suggestions to improve the clarity of the presentation.
This paper is part of my Ph.D. thesis at  University of S\~ao Paulo, which was done under the guidance of Professor Jaime Angulo to whom I express my thanks. The author gratefully acknowledges the support from CNPq, through grant No. 152672/2016-8.

\section{Appendix}
\label{S:5}

We list some properties of the Orlicz space $L^{A}(\mathbb{R})$ that we have used above.  For a proof of such statements we refer to \cite[Lemma 2.1]{CL}.

\begin{proposition} \label{orlicz}
Let $\left\{u_{{m}}\right\}$ be a sequence in  $L^{A}(\mathbb{R})$, the following facts hold:\\
{\it i)} If  $u_{{m}}\rightarrow u$ in $L^{A}(\mathbb{R})$, then $A(\left|u_{{m}}\right|)\rightarrow A(\left|u\right|)$ in $L^{1}(\mathbb{R})$ as   $n\rightarrow \infty$.\\
{\it ii)} Let  $u\in L^{A}(\mathbb{R})$. If  $u_{m}(x)\rightarrow u(x)$ $a.e.$  $x\in\mathbb{R}$ and if
\begin{equation*}
\lim_{n \to \infty}\int_{\mathbb{R}}A\left(\left|u_{m}(x)\right|\right)dx=\int_{\mathbb{R}}A\left(\left|u(x)\right|\right)dx,
\end{equation*}
then $u_{{m}}\rightarrow u$ in $L^{A}(\mathbb{R})$ as   $n\rightarrow \infty$.\\
{\it iii)} For any $u\in L^{A}(\mathbb{R})$, we have
\begin{equation}\label{DA1}
{\rm min} \left\{\left\|u\right\|_{L^{A}},\left\|u\right\|^{2}_{L^{A}}\right\}\leq  \int_{\mathbb{R}}A\left(\left|u(x)\right|\right)dx\leq {\rm max} \left\{\left\|u\right\|_{L^{A}},\left\|u\right\|^{2}_{L^{A}}\right\}.
\end{equation}
\end{proposition}


\medskip
Received July 2016; revised February 2017.
\medskip

\end{document}